\documentclass{amsart} 
\usepackage{epsfig}
\usepackage{wrapfig}
\usepackage[bookmarks,colorlinks,breaklinks]{hyperref}
\hypersetup{
    colorlinks,%
    citecolor=black,%
    filecolor=black,%
    linkcolor=black,%
    urlcolor=blue
}
\calclayout
\newcommand{\BEA}{\begin{eqnarray}}
\newcommand{\EEA}{\end{eqnarray}}
\usepackage{lipsum}
\usepackage{bm}
\usepackage{amsfonts}
\usepackage{graphicx}
\usepackage{epstopdf}
\usepackage[ruled,vlined]{algorithm2e}
\usepackage{slashbox}
\usepackage{ascii}
\newtheorem{assumption}{Assumption}
\newtheorem{examp}{Example} 
\newcommand{\comment}[1]{}
\def \bes{\begin{eqnarray}}
\def \ees{\end{eqnarray}}
\def \bns{\begin{eqnarray*}}
\def \ens{\end{eqnarray*}}

\newtheorem{theorem}{Theorem}[section]

\SetKwRepeat{Do}{do}{while}%

\begin{document}

\title{An Equation-By-Equation Method for Solving the Multidimensional Moment Constrained Maximum Entropy Problem}
\author{Wenrui Hao}
\address{Department of Mathematics, the Pennsylvania State University, University Park, PA}
\email{wxh64@psu.edu}
\thanks{The research of W.H. was partially supported by the American Heart Association under Grant 17SDG33660722 and the Institute
for CyberScience Seed Grant.}

\author{John Harlim}
\address{Department of Mathematics, Department of Meteorology and Atmospheric Science, the Pennsylvania State University, University Park, PA}
\email{jharlim@psu.edu}
\thanks{The research of J.H. was partially supported by the Office of Naval Research Grants N00014-16-1-2888, MURI N00014-12-1-0912 and the National Science Foundation Grants DMS-1317919, DMS-1619661.}

\subjclass[2010]{65H10, 65H20, 94A17, 49M15}

\keywords{Homotopy continuation, moment constrained, maximum entropy, equation-by-equation method}

\date{\today}

\dedicatory{}

\begin{abstract}
An equation-by-equation (EBE) method is proposed to solve a system of nonlinear equations arising from the moment constrained maximum entropy problem of multidimensional variables. The design of the EBE method combines ideas from homotopy continuation and Newton's iterative methods. Theoretically, we establish the local convergence under appropriate conditions and show that the proposed method, geometrically, finds the solution by searching along the surface corresponding to one component of the nonlinear problem. We will demonstrate the robustness of the method on various numerical examples, including: (1) A six-moment one-dimensional entropy problem with an explicit solution that contains components of order $10^0-10^3$ in magnitude; (2) Four-moment multidimensional entropy problems with explicit solutions where the resulting systems to be solved ranging from $70-310$ equations; (3) Four- to eight-moment of a two-dimensional entropy problem, which solutions correspond to the densities of the two leading EOFs of the wind stress-driven large-scale oceanic model. In this case, we find that the EBE method is more accurate compared to the classical Newton's method, the MATLAB generic solver, and the previously developed BFGS-based method, which was also tested on this problem. (4) Four-moment constrained of up to five-dimensional entropy problems which solutions correspond to multidimensional densities of the components of the solutions of the Kuramoto-Sivashinsky equation. {\color{black}For the higher dimensional cases of this example, the EBE method is superior because it automatically selects a subset of the prescribed moment constraints from which the maximum entropy solution can be estimated within the desired tolerance.} This selection feature is particularly important since the moment constrained maximum entropy problems do not necessarily have solutions in general.
\end{abstract}

\maketitle

\section{Introduction}\label{section1}
The maximum entropy principle provides a natural criterion for estimating the least biased density function subjected to the given moments \cite{jaynes:57}. This density estimation approach has a wide range of applications, such as, the harmonic solid and quantum spin systems \cite{mp:84}, econometrics \cite{wu:03}, and geophysical applications \cite{amk:05,hma:05}. In a nutshell, this moment constrained method is a parametric estimation technique where the resulting density function is in the form of an exponential of polynomials. This is a consequence of maximizing the Shannon entropy subjected to the polynomial moment constraints, which is usually transformed into an unconstrained minimization problem of a Lagrangian function \cite{wptz:01}. Standard approaches for solving this unconstrained minimization problem are based on Newton's iterative method \cite{abramov2007,wptz:01} or quasi-Newton's based method such as the BFGS method \cite{abramov2009,abramov2010}.

In the last two papers \cite{abramov2009,abramov2010}, where the BFGS-based method was introduced and reviewed, they considered minimization problems that involve 44-83 equations, resulting from a 2D problem with moment constraints of up to order-eight, a 3D problem with moment constraints of up to order-six, and a 4D problem with moment constraints of up to order-four. In this paper, we introduce a novel equation solver that can be used to find density function of moderately high dimensional problems (e.g., systems of 70-310 equations resulting from moments up to order-four of 4-7 dimensional density functions) provided that the solutions exist. The proposed method, which we called the Equation-By-Equation (EBE) method, is an iterative method that solves a one-dimensional problem at the first iterate, a two-dimensional problem at the second iterate, a three-dimensional problem at the third iterate, and eventually, solves the full system of nonlinear equations corresponding to the maximum entropy problem at the last iterate. Technically, this method combines Newton's method with ideas from homotopy continuation. We will show that the EBE method is locally convergent under appropriate conditions. Furthermore, we will provide sufficient conditions for global convergence. Through the convergence analysis, we will show that, geometrically, the proposed method finds the solution of the nonlinear system of equations by tracking along the surface corresponding to one component of the system of nonlinear equations. The EBE method automatically selects a subset of the prescribed constraints from which the maximum entropy solution can be estimated within the desired tolerance. This is an important feature since the maximum entropy problems do not necessarily have solutions for general set of moment constraints.

We shall find that the EBE method produces more accurate solutions (smaller error in the moments) compared to the classical Newton's method, the MATLAB built-in {\asciifamily fsolve.m}, and BFGS method on the test problem in \cite{abramov2009,abramov2010} and on test problems based on the solutions of the Kuramoto-Shivashinski equation. Numerically, we will demonstrate that the EBE method is able to solve problems where the true solutions consist of components of order $10^0-10^3$. We shall also see that the EBE method can solve a system of hundreds of equations in various examples, including those with explicit solutions as well as those with densities estimated based on solutions of complex spatially extended dynamical systems.

The remaining part of the paper is organized as follows. In Section~\ref{section2}, we give a brief overview of the multidimensional maximum entropy problem. In Section~\ref{section3}, we introduce the EBE algorithm. In Section~\ref{section4}, we provide the local convergence analysis. In Section~\ref{section5}, we discuss the practical issues with the proposed method and provide remedies. In Section~\ref{section6}, we demonstrate the robustness of the EBE method on various numerical examples. In Section~\ref{section7}, we conclude the paper with a brief summary and discussion. We include an Appendix to show some computational details that are left out in the main text. {\color{black}Interested readers and users can access the EBE codes (written in MATLAB) at \cite{ebecodes}.}

\section{An overview of the maximum entropy problem}\label{section2}

We consider the Haussdorf moment-constrained maximum entropy problem \cite{abramov2007,abramov2010,ft1994}. That is, find the optimal probability
density $\rho^*({\bf x})$ which maximizes the Shannon entropy,
\bes S(\rho):=-\int_{\Omega}\log(\rho(\mathbf{x}))\rho({\bf x}) d\mathbf{x},\label{Max}\ees
where $\mathbf{x}\in\Omega=[-1,1]^d$ satisfies the following linear constraints,
\bes\mathcal{F}_{\bf j}:=\int_{\Omega}c_{\bf j}(\mathbf{x})\rho(\mathbf{x})d\mathbf{x}=f_{\bf j}, \quad |{\bf j}|=0,1,2,\cdots,p.\label{Cons}\ees
In applications, one usually computes the statistics $f_{\bf j}$ from samples of data. For arbitrary finite domain, one can rescale the data to the domain $\Omega$.

While $c_{\bf j}(\mathbf{x})$ can be arbitrary functions in $L^1(\Omega,\rho)$, we will focus on the usual uncentered statistical moments with monomial basis functions, $c_{\bf j}(\mathbf{x})={\bf x}^{\bf j}$ in this article, where we have adopted the notations ${\bf x}=(x_1,\ldots,x_d)\in\Omega$, ${\bf j}=(j_1,\ldots,j_d)\in\mathbb{Z}_+^d$ with $\mathbb{Z}_+ = \{0,1,2,\ldots\}$, and ${\bf x}^{\bf j} = \prod_{i=1}^d x_i^{j_i}$. In \eqref{Cons}, the quantities $f_{\bf j}$ are the given ${\bf j}$-th moments that can be computed from the data. Since the total number of monomials ${\bf x}^{\bf j}$ where $|{\bf j}| = j$ is $C^{j+d-1}_{d-1}$, then the total number of constraints in \eqref{Cons} for moments up to order-$p$ is,
\bes
n = \sum_{j=1}^p C^{j+d-1}_{d-1},\nonumber
\ees
excluding the normalization factor corresponding to $c_{\bf 0}({\bf x})=1$. For example, in two-dimensional problem, the total number of moments up to order $p=4$ is $n=14$. To simplify the notation below, we will use a single index notation and understood that the total number of constraints to be satisfied is $n$, excluding the zeroth moment. The exclusion of the zeroth moment will be clear as we discuss below.

By introducing Lagrange multipliers, the above constrained optimization problem can be transformed into the following unconstrained problem:
\bes \mathcal{L}(\rho(\mathbf{x}),\lambda_0,\cdots\lambda_n)=S(\rho)+\sum_{j=0}^n\lambda_j(\mathcal{F}_j-f_j).\label{Lag}\ees
In order to find a solution of  \eqref{Lag}, we set $\frac{\partial \mathcal{L}}{\partial \rho}=0$, which gives,
\bes
 \rho(\mathbf{x})=\frac{1}{Z}\exp\Big({\sum_{j=1}^n\lambda_jc_j(\mathbf{x})\Big)},\label{expoly}
\ees
where we have defined $Z = \exp (1-\lambda_0)$.
Since $\int_{\Omega} \rho(\mathbf{x})d\mathbf{x}=1$, we have
\bes Z(\lambda_1,\ldots,\lambda_n) =\int_{\Omega} \exp \Big({\sum_{j=1}^n\lambda_jc_j(\mathbf{x})}\Big)\,d\mathbf{x},\label{norm}\ees
which indicates that $Z$ (or implicitly $\lambda_0$) is a function of $\lambda_1,\ldots, \lambda_n$. Therefore, the normalization factor $Z$ can be computed via \eqref{norm} once $\lambda_1,\ldots, \lambda_n$ are estimated. Therefore, we can just concentrate on finding the Lagrange multipliers $\lambda_1,\ldots, \lambda_n$ which satisfy $n$ constraints in \eqref{Cons}, excluding the case $c_{\bf{0}}({\bf x})=1$. In particular, the constrained maximum entropy problem is to solve the following nonlinear system of integral equations,
\bes F_j(\lambda_1,\cdots,\lambda_n)&:=& \mathcal{F}_{j}(\lambda_1,\ldots,\lambda_n)-f_j \nonumber \\ &=&  \int_{\Omega} (c_j(\mathbf{x})-f_j)\exp\Big({\sum_{k=1}^n\lambda_kc_k(\mathbf{x})}\Big)\,d\mathbf{x}=0,\quad j=1,\ldots, n,\label{Poly}\ees
for $\lambda_1,\ldots,\lambda_n$.

%
In our numerical implementation, the integral in system (\ref{Poly}) will be approximated with a nested sparse grid quadrature rule \cite{GG},
\[\int_\Omega f(\mathbf{x})d\mathbf{x} \approx \sum_i f(\mathbf{x}_i) w_i,\]
where $\mathbf{x}_i$ are the nested sparse grid nodes, and $w_i$ are the corresponding weights based on the nested Clenshaw-Curtis quadrature rule \cite{trefethen2008}. The number of nodes depends on the dimension of the problem $d$ and the number of the nested set (based on the Smolyak construction \cite{smolyak1963}), is denoted with the parameter $\ell$ (referred as level). In the numerical implementation, we need to specify the parameter $\ell$.

\section{An equation-by-equation algorithm}\label{section3}
In this section, we describe the new Equation-By-Equation (EBE) technique to solve the system of equations in \eqref{Poly},
\bes
{\bf F}_{n}(\bm{\lambda}_n)={\bf 0},\label{Fn}
\ees
where we have defined,
\bes
\mathbf{F}_{n}(\bm{\lambda}_n):=\Big(F_1(\bm{\lambda}_n),\ldots, F_{n}(\bm{\lambda}_n)\Big),\nonumber
\nonumber\ees
and $\bm{\lambda}_n = (\lambda_i,\ldots,\lambda_n)$. In the following iterative scheme, we start the iteration with an initial condition $(\alpha_1,\ldots,\alpha_n)\in\mathbb{R}^n$. We define $\bm{\mu}^{(i)}\in\mathbb{R}^i$ as the exact solution to the following $i$-dimensional system,
\bes
\mathbf{F}_{i}(\bm{\lambda}_i,\alpha_{i+1},\ldots,\alpha_n) = {\bf 0}, \quad i=1,\ldots, n,\label{Fi}
\ees
where we have fixed the last $n-i$ coefficients, $\lambda_{i+1}=\alpha_{i+1}$, \ldots,  $\lambda_{n}=\alpha_{n}$. With this notation, the exact solution for \eqref{Fn} is $\bm{\mu}^{(n)}\in\mathbb{R}^n$. We also define $\bm{\hat{\mu}}^{(i)}$ to be the numerical estimate of $\bm{\mu}^{(i)}$. With these notations, we now describe the algorithm.

Generally speaking, at each iteration-$i$, where $i=1,\ldots,n$, the EBE algorithm solves a system of $i$-dimensional system in \eqref{Fi}. At each step-$i$, given the numerical solution at the previous step, $\bm{\hat{\mu}}^{(i-1)} \in\mathbb{R}^{i-1}$ and initial condition $\alpha_i$, we apply idea from homotopy continuation to find the solution $\bm{\mu}^{(i)}\in\mathbb{R}^i$ that solves the $i$-dimensional system of equations \eqref{Fi}. Notice that we do not only add a new equation $F_{i}(\bm{\lambda}_{i},\alpha_{i+1},\ldots,\alpha_n)=0$ but we also estimate the $i$th variable in the previous $(i-1)$ equations,
$\mathbf{F}_{i-1}(\bm{\lambda}_{i},\alpha_{i+1},\ldots,\alpha_{n})={\bf 0}$. The scheme proceeds by solving the larger systems one-by-one until $i=n$ so we eventually solve \eqref{Fn}.

Now let us describe how to numerically estimate $\bm{\mu}^{(i)}$ at every step-$i$. For the first step $i=1$, we solve the one-dimensional problem,
\bes
{\bf F}_1(\bm{\lambda}_1,\alpha_2,\ldots,\alpha_n) = 0, \nonumber
\ees
for $\bm{\lambda}_1$ with Newton's method. For the steps $i = 2,\ldots,n$, we have $\bm{\hat{\mu}}^{(i-1)}$ which are the numerical estimates of ${\bf F}_{i-1}(\bm{\lambda}_{i-1},\alpha_{i},\ldots,\alpha_n)={\bf 0}$. To simplify the expression below, let us use $F_{i}(\bm{\lambda}_{i-1},\lambda_{i})$ as a short hand notation for $F_{i}(\bm{\lambda}_{i-1},\lambda_{i},\alpha_i,\ldots,\alpha_n)$ to emphasize the independent variables.

We proceed to estimate $\lambda_{i}$ using Newton's method with $Tol_1$ on the $i$-th equation. That is, we iterate
\bes
\lambda_{i}^{m+1}&=&\lambda_{i}^{m}-\Big(\frac{\partial F_{i}}{\partial \lambda_{i}}(\bm{\lambda}_{i-1}^{m},\lambda_{i}^{m})\Big)^{-1} F_{i}(\bm{\lambda}_{i-1}^{m},\lambda_{i}^{m}) , \quad m=0,1\ldots,\label{scalarnewton}\\
 \lambda_i^0 &=& \alpha_i, \quad \bm{\lambda}_{i-1}^0 = \bm{\hat{\mu}}^{(i-1)} \nonumber
\ees
assuming that $\frac{\partial F_{i}}{\partial \lambda_{i}}(\bm{\lambda}_{i-1}^{m},\lambda_{i}^{m})\neq 0$. Here, the partial derivative of $F_{i}$ with respect to $\lambda_{i}$ evaluated at $\lambda_{i}^{m}$  is defined as,
\bes
\frac{\partial F_{i}}{\partial \lambda_{i}}(\bm{\lambda}_{i-1}^{m},\lambda_{i}^{m}) =\int_{\Omega} (c_{i}(\mathbf{x})-f_{i})c_{i}(\mathbf{x})\exp\Big({\sum_{j=1}^{i-1}\lambda_{j}^{m} c_j(\mathbf{x})+ \lambda_{i}^{m}c_{i}(\mathbf{x})}\Big)\,d\mathbf{x},
\ees
where we have denoted $\bm{\lambda}_{i-1}^m = (\lambda_{i}^{m},\ldots,\lambda_{i-1}^{m})$. Notice that to proceed the iteration in \eqref{scalarnewton}, we need to update $\bm{\lambda}_{i-1}^{m}$ for $m>0$. We propose to follow the homotopy continuation method for this update. In particular, we are looking for $\bm{\lambda}_{i-1}^{m+1}$ that solves $\mathbf{F}_{i-1}(\bm{\lambda}_{i-1}^{m+1},\lambda_{i}^{m+1})={\bf 0}$, given the current estimate $\lambda_{i}^{m+1}$ from \eqref{scalarnewton} as well as $\mathbf{F}_{i-1}(\bm{\lambda}_{i-1}^{m},\lambda_{i}^{m})={\bf 0}$. At $m=0$, this last constraint is numerically estimated by $\mathbf{F}_{i-1}(\bm{\hat{\mu}}^{(i-1)},\alpha_{i})\approx{\bf 0}$.

One way to solve this problem is through the following predictor-corrector step which is usually used in homotopy continuation method \cite{BHSW,SW}. In particular, we apply Taylor's expansion to
\bes
\mathbf{F}_{i-1}(\bm{\lambda}_{i-1}^{m+1},\lambda_{i}^{m+1}) = \mathbf{F}_{i-1}(\bm{\lambda}_{i-1}^{m}+\Delta \bm{\lambda},\lambda_{i}^{m} + (\lambda_{i}^{m+1}-\lambda_{i}^{m}))={\bf 0}\nonumber
\ees
at $(\bm{\lambda}_{i-1}^{m},\lambda_{i}^{m})$, which gives,
\bes
\mathbf{F}_{i-1}(\bm{\lambda}_{i-1}^{m},\lambda_{i}^{m}) + {\bf F}_{i-1,\bm{\lambda}_{i-1}}(\bm{\lambda}_{i-1}^{m},\lambda_{i}^{m}) \Delta \bm{\lambda} + {\bf F}_{i-1,\lambda_{i}}(\bm{\lambda}_{i-1}^{m},\lambda_{i}^{m}) (\lambda_{i}^{m+1}-\lambda_{i}^{m}) ={\bf 0},\nonumber
\ees
which means that
\bes
\Delta \bm{\lambda} = -{\bf F}^{-1}_{i-1,\bm{\lambda}_{i-1}}(\bm{\lambda}_{i-1}^{m},\lambda_{i}^{m}) {\bf F}_{i-1,\lambda_{i}}(\bm{\lambda}_{i-1}^{m},\lambda_{i}^{m})(\lambda_{i}^{m+1}-\lambda_{i}^{m}),\nonumber
\ees
assuming that ${\bf F}_{i-1,\bm{\lambda}_{i-1}}(\bm{\lambda}_{i-1}^{m},\lambda_{i}^{m})$ is invertible. Based on this linear prediction, $\bm{\lambda}_{i-1}^{m+1}$ is approximated by,
\bes
\bm{\tilde\lambda}_{i-1}^{m+1} &=& \bm{\lambda}_{i-1}^{m}+\Delta \bm{\lambda} \nonumber \\ &=&   \bm{\lambda}_{i-1}^{m}-{\bf F}^{-1}_{i-1,\bm{\lambda}_{i-1}}(\bm{\lambda}_{i-1}^{m},\lambda_{i}^{m}) {\bf F}_{i-1,\lambda_{i}}(\bm{\lambda}_{i-1}^{m},\lambda_{i}^{m})(\lambda_{i}^{m+1}-\lambda_{i}^{m}) .\label{predictor}
\ees
Subsequently, when $\|\mathbf{F}_i(\bm{\tilde\lambda}_{i-1}^{m+1},\lambda_{i}^{m+1})\|\geq Tol_2$ , apply a correction using Newton's method by expanding,
\bes
{\bf 0} = \mathbf{F}_{i-1}(\bm{\lambda}_{i-1}^{m+1},\lambda_{i}^{m+1}) =\mathbf{F}_{i-1}(\bm{\tilde\lambda}_{i-1}^{m+1},\lambda_{i}^{m+1}) +  \mathbf{F}_{i-1,\bm{\lambda}_{i-1}}(\bm{\tilde\lambda}_{i-1}^{m+1},\lambda_{i}^{m+1})\Delta \bm{\tilde\lambda},\nonumber
\ees
assuming that $\bm{\lambda}_{i-1}^{m+1}= \bm{\tilde\lambda}_{i-1}^{m+1}+\Delta \bm{\tilde\lambda}$,
to find that,
\bes
\bm{\lambda}_{i-1}^{m+1}= \bm{\tilde\lambda}_{i-1}^{m+1}  - \mathbf{F}_{i-1,\bm{\lambda}_{i-1}}(\bm{\tilde\lambda}_{i-1}^{m+1},\lambda_{i}^{m+1})^{-1} \mathbf{F}_{i-1}(\bm{\tilde\lambda}_{i-1}^{m+1},\lambda_{i}^{m+1}).\label{corrector}
\ees
This expression assumes that $\mathbf{F}_{i-1,\bm{\lambda}_{i-1}}(\bm{\tilde\lambda}_{i-1}^{m+1},\lambda_{i}^{m+1})$ is invertible.

In summary, at each step-$i$, we iterate \eqref{scalarnewton}, \eqref{predictor}, \eqref{corrector}. So, the outer loop $i$ corresponds to adding one equation to the system at the time and for each $i$, we apply an inner loop, indexed with $m$, to find the solution $\bm{\mu}^{(i)}$ for ${\bf F}_i(\bf{\lambda}_i,\alpha_{i+1},\ldots,\alpha_n)={\bf 0}$. We denote the approximate solution as $\bm{\hat\mu}^{(i)}$. An adaptive tolerance technique is employed to compute the initial guess of ${\bf F}_i$  by using Newton's method. In particular, when the current tolerance $Tol_2$ is not satisfied after executing \eqref{corrector}, then we divide $Tol_1$ by ten until $Tol_2$ is met.

\comment{
The pseudocode is summarized in Algorithm~\ref{alg1}.
\begin{algorithm}[H]
\small\SetAlgoLined
 \SetAlgorithmName{Algorithm}{problem}{List of problems}
\SetKwInOut{Input}{Input}\SetKwInOut{Output}{Output} 
\Input{Degree index for $c_j(\mathbf{x})$, moments $f_j$, $j=1,\cdots,n$, tolerance for homotopy $Tol_1$ and tolerance for Newton's method $Tol_2$} \Output{A solution $\Lambda_n$}
Set ${\bm{\alpha}}=0$;\\ \For{$i=1,\cdots,n$}{
Set $Tol=Tol_1$;\\
\Do{Newton's method succeeds}{
Solve
$\mathbf{F}_i$ by using Newton's method in \eqref{predictor} with $\hat{\mathbf{\Lambda}}_i$
as the initial guess, upload $\mathbf{\Lambda}_i$;\\
\If{$\| \mathbf{F}_i(\bm{\tilde\lambda}_{i-1}^{m+1},\lambda_{i}^{m+1})\|\geq Tol_2$}{
$Tol=Tol/10$;\\
Update $\hat{\mathbf{\Lambda}}_i$  with $\|F_{i}(\lambda_{i}^{m+1})\|<Tol$;}
}
\If{i=n}{break;} Set $\lambda_{i+1}^0=0$ and $m=1$;
\\\While{$\|F_{i+1}(\lambda_{i+1}^m)\|<Tol_1$}{$\lambda_{i+1}^{m+1}=\lambda_{i+1}^m-F_{i+1}(\lambda_{i+1}^m)/dF_{i+1}(\lambda_{i+1}^m)$\\Adaptively homotopy $\mathbf{F}_i$ from $\lambda_{i+1}^m$ to $\lambda_{i+1}^{m+1}$ and obtain the solution $\mathbf{\Lambda}_i^{m+1}$;\\m=m+1;}
Update $\hat{\mathbf{\Lambda}}_i=[\mathbf{\Lambda}_i;\lambda_{i+1}^m]$;}
 \caption{Summary of the equation-by-equation algorithm}\label{alg1}
\end{algorithm}}

Recall that the standard Newton's method assumes that the Jacobian ${\bf F}_{n,\bm{\lambda}_n}\in\mathbb{R}^{n\times n}$ is nonsingular at the root of the full system in \eqref{Poly} to guarantee the local convergence. 
In the next section, we will show that the  EBE method requires the following conditions for local convergence:
\begin{assumption}
Let $\bm{\mu}^{(i)}\in\mathbb{R}^i$ be a solution of ${\bf F}_{i}(\bm{\lambda}_{i},\alpha_{i+1},\ldots,\alpha_n)={\bf 0}$, for each $i=1,\ldots,n$. The EBE method assumes the following conditions:
\begin{enumerate}
\item $\frac{\partial F_{i}}{\partial \lambda_{i}}(\bm{\mu}^{(i)},\alpha_{i+1},\ldots,\alpha_n)\neq 0$.
\item $\mathbf{F}_{i,\bm{\lambda}_{i}}(\bm{\mu}^{(i)},\alpha_{i+1},\ldots,\alpha_n)$ are nonsingular.
\item Each component of ${\bf F_{i}}$ is twice differentiable in a close region whose interior contains the solution $\bm{\mu}^{(i)}$.
\end{enumerate}
\end{assumption}

These conditions are similar to the standard Newton's assumptions on each system of $i$ equations. The smoothness condition will be used in the proof of the local convergence in the next section. Of course if one can specify initial conditions that are sufficiently close to the true solution, then one can simply apply Newton's method directly. With the EBE method, we can start with any arbitrary initial condition. Theoretically, this will require an additional condition beyond the Assumption~1 for global convergence as we shall discuss in Section~\ref{section4}. In Section~\ref{section5}, we will provide several remedies when the initial condition is not close to the solution. In fact, we will always set the initial condition to zero in our numerical implementation in Section~\ref{section6}, $\alpha_i=0, \forall i=1,\ldots, n$, and demonstrate that the EBE method is numerically accurate in the test problems with solutions that are far away from zero.

\section{Convergence analysis}\label{section4}

In this section, we study the convergence of this method. First, let's concentrate on the convergence of the iteration \eqref{scalarnewton}, \eqref{predictor}, \eqref{corrector} for solving the $i$-dimensional system, ${\bf F}_i(\bm{\lambda}_{i-1},\lambda_i,\alpha_{i+1},\ldots,\alpha_n):= {\bf F}_i(\bm{\lambda}_{i-1},\lambda_i)={\bf 0}$ for $\bm{\lambda}_{i-1}$ and $\lambda_i$.
In compact form, these three steps can be written as an iterative map,
\bes
(\bm{\lambda}_{i-1}^{m+1},\lambda_{i+1}^{m+1}) = {\bf H}_{i} (\bm{\lambda}_{i-1}^{m},\lambda_{i}^{m}),\label{iterateH}
\ees
where the map ${\bf H}_{i}:\mathbb{R}^{i}\to \mathbb{R}^{i}$ is defined as,
\bes
{\bf H}_{i}(\bm{\lambda}_{i-1}, \lambda_{i}) := \begin{pmatrix}  \bm{g}_{i}  - \mathbf{F}_{i-1,\bm{\lambda}_{i-1}}(\bm{g}_{i},H_{i,2})^{-1} \mathbf{F}_{i-1}(\bm{g}_{i},H_{i,2}) \\
\lambda_{i}-(\frac{\partial F_{i}}{\partial \lambda_{i}}(\bm{\lambda}_{i-1},\lambda_{i}))^{-1} F_{i}(\bm{\lambda}_{i-1},\lambda_{i}) \end{pmatrix}.\label{mapH}
\ees
In \eqref{mapH}, the notation $H_{i,2}$ denotes the second component of \eqref{mapH} and
\bes
\bm{g}_{i} :=  \bm{\lambda}_{i-1} -{\bf F}_{i-1,\bm{\lambda}_{i-1}}(\bm{\lambda}_{i-1},\lambda_i)^{-1} {\bf F}_{i-1,\lambda_{i}}(\bm{\lambda}_{i-1},\lambda_i)(H_{i,2}-\lambda_{i}) \label{functiong}
\ees
is defined exactly as in \eqref{predictor}.

For notational convenience in the discussion below, we let the components of the exact solution of \eqref{Fi} be defined as $\bm{\mu}^{(i)}:=(\bm{\mu}_{i-1}^{(i)}, \mu_i^{(i)})\in\mathbb{R}^{i}$. Here, we denote the first $i-1$ components as $\bm{\mu}_{i-1}^{(i)}=(\mu_1^{(i)},\ldots,\mu_{i-1}^{(i)})\in\mathbb{R}^{i-1}$. Similarly, we also denote ${\bf H}_{i} = ({\bf H}_{i,1},{H}_{i,2})$.
First, we can deduce that,

\begin{theorem}
Let $\bm{\mu}^{(i)} \in\mathbb{R}^{i}$ be a fixed point of \eqref{iterateH}. Assume that ${\bf F}^*_{i-1,\bm{\lambda}_{i-1}}:={\bf F}_{i-1,\bm{\lambda}_{i-1}}(\bm{\mu}^{(i)})$ is nonsingular and $\frac{\partial F^*_{i}}{\partial \lambda_{i}}:= \frac{\partial F_{i}}{\partial_{\lambda_{i}}} (\bm{\mu}^{(i)})\neq 0$, then ${\bf F}^*_{i} :={\bf F}_{i}(\bm{\mu}^{(i)})={\bf 0}$.
\end{theorem}

\begin{proof}
Evaluating the second equation in \eqref{mapH} at the fixed point, we obtain
\bes
\mu_i^{(i)} = \mu_i^{(i)}-\Big(\frac{\partial F^*_{i}}{\partial \lambda_{i}}\Big)^{-1} F^*_{i},\nonumber
\ees
which means that $F^*_{i}:=F_{i}(\bm{\mu}^{(i)})=0$. This also implies that $H^*_{i,2}=\mu_{i}^{(i)}$, where $H^*_{i,2}$ denotes the second component of \eqref{mapH} evaluated at the fixed point. Subsequently,
\bes
\bm{g}^*_{i} := \bm{g}_{i}(\bm{\mu}_{i-1}^{(i)},\mu_i^{(i)}) =  \bm{\mu}_{i-1}^{(i)}.\nonumber
\ees
Substituting $H^*_{i,2}=\mu_i^{(i)}$ and $\bm{g}^*_{i}=\bm{\mu}_{i-1}^{(i)}$ into $\bm{\mu}_{i-1}^{(i)}=\bm{H}_{i,1}^*$, where $\bm{H}^*_{i,1}$ denotes the first equation in \eqref{mapH} evaluated at the fixed point $\bm{\mu}^{(i)}$, we immediately obtain $\mathbf{F}^*_{i-1}:=\mathbf{F}_{i-1}(\bm{\mu}^{(i)}) = {\bf 0}$ and the proof is completed.
\end{proof}

This theorem says that the fixed points of \eqref{iterateH} are indeed the solutions of \[{\bf F}_i(\bm{\lambda}_{i-1},\lambda_i,\alpha_{i+1},\ldots,\alpha_n)={\bf 0},\] which is what we intend to solve on each iteration $i=2,\ldots, n$. Next, we will establish the condition for the fixed point to be locally attracting. This condition will ensure that if we iterate the map in \eqref{mapH} with an initial condition that is close to the solution, then we will obtain the solution.

For local convergence, we want to show that eigenvalues of the Jacobian matrix $D{\bf H}_{i}^*:=D{\bf H}_{i}(\bm{\mu}^{(i)})$ are in the interior of the unit ball of the complex plane. One can verify that the components of the Jacobian matrix $D{\bf H}_{i}^*$ are given by,
\bes
\frac{\partial \bm{H}_{i,1}^*}{\partial\lambda_j} &=& -(\mathbf{F}_{i-1,\bm{\lambda}_{i-1}}^*)^{-1} \mathbf{F}_{i-1,\lambda_{i}}^*\frac{\partial H_{i,2}^*}{\partial\lambda_j},\label{H1}\\
\frac{\partial H_{i,2}^*}{\partial\lambda_j} &=& \delta_{j,i} - \Big(\frac{\partial F_{i}^*}{\partial \lambda_{i}}\Big)^{-1}\frac{\partial F_{i}^*}{\partial \lambda_j},\label{H2}
\ees
for $j=1,\ldots,i$, where we have used all the three conditions in the Assumption~1 (see Appendix~\ref{Jac} for the detailed derivation).  Here, $\delta_{j,i}$ is one only if $j=i$ and zero otherwise. To simplify the discussion below, let's define the following notations,
\bes
J &:=& \mathbf{F}_{i-1,\bm{\lambda}_{i-1}}^*\nonumber\\
{\bf v} &:=& \mathbf{F}_{i-1,\lambda_{i}}^* \label{notations}\\
{\bf c} &:=& \Big(\frac{\partial H_{i,2}^*}{\partial\lambda_1},\ldots, \frac{\partial H_{i,2}^*}{\partial\lambda_{i-1}}\Big)^\top\nonumber
\ees
such that,
\bes
D{\bf H}_{i+1}^*= \begin{pmatrix}  J^{-1}{\bf v}{\bf c}^\top & \vec{0} \\  {\bf c}^\top & 0 \end{pmatrix}\in\mathbb{R}^{i\times i}.\label{DH}
\ees

We can now obtain the following result:
\begin{theorem}\label{thm2}
Let $\bm{\mu}^{(i)}\in\mathbb{R}^{i}$ be a fixed point of \eqref{iterateH} such that the conditions in the Assumption~1 are satisfied. Let's $\sigma_j({\bm F}_{i-1,{\bm\lambda}_{i-1}}^*)$ be the eigenvalues of $\mathbf{F}_{i-1,{\bm\lambda}_{i-1}}^*$ and assume that they satisfy the following order $|\sigma_1|\geq |\sigma_2|\geq \ldots |\sigma_{i-1}|$. If
\begin{equation}
\Big|\Big(\frac{\partial F_{i}^*}{\partial \lambda_{i}}\Big)^{-1} \sum_{j=1}^{i-1} \frac{\partial F_j^*}{\partial \lambda_{i}} \frac{\partial F_{i}^*}{\partial \lambda_j}\Big| < |\sigma_{i-1}(\mathbf{F}_{i-1,\bm{\lambda}_{i-1}}^*)|,\label{assumption}
\end{equation}
 then $\bm{\mu}^{(i)}$ is locally attracting.
\end{theorem}

\begin{proof}
From  \eqref{DH}, we only need to analyze the eigenvalues of $J^{-1}{\bf v}{\bf c}^\top$. From basic matrix theory, recall that the magnitude of the largest eigenvalue can be bounded above as follows,
\bes
|\sigma_1(J^{-1}{\bf v}{\bf c}^\top)| = \|J^{-1}{\bf v}{\bf c}^\top\|_2 \leq \|J^{-1}\|_2 \| {\bf v}{\bf c}^\top \|_2,\nonumber
\ees
where $\|\cdot\|_2$ denotes the matrix $\ell_2$-norm. For the fixed point to be locally attracting, all of the eigenvalues of $J^{-1}{\bf v}{\bf c}^\top$ have to be in the interior of the unit ball in the complex plane. This means that we only need to show that $\|J^{-1}\|_2 \| {\bf v}{\bf c}^\top \|_2 < 1$ or $\| {\bf v}{\bf c}^\top \|_2 < |\sigma_{i-1}(J)|$, where $\sigma_{i-1}(J)$ denotes the smallest eigenvalue of the $(i-1)\times (i-1)$ matrix $J$ following the ordering in the hypothesis.

Since $\mbox{Tr}({\bf v}{\bf c}^\top) = \sum_{j=1}^i \sigma_j({\bf v}{\bf c}^\top)$ and ${\bf v}{\bf c}^\top$ is a rank-one matrix, then its nontrivial eigenvalue is given by,
\bes
\sigma({\bf v}{\bf c}^\top) = \mbox{Tr}({\bf v}{\bf c}^\top) = \sum_{j=1}^{i-1} \frac{\partial F_{j}^*}{\partial \lambda_{i}} \frac{\partial H_{i,2}^*}{\partial\lambda_j} = - \sum_{j=1}^{i-1} \frac{\partial F_j^*}{\partial \lambda_{i}} \frac{\partial F_{i}^*}{\partial \lambda_j}\Big(\frac{\partial F_{i}^*}{\partial \lambda_{i}}\Big)^{-1},\nonumber
\ees
where we have used the definitions in \eqref{notations} and the second component in \eqref{H2}. From the assumption in \eqref{assumption}, we have
\bes
 \| {\bf v}{\bf c}^\top \|_2 = |\sigma({\bf v}{\bf c}^\top)| = \Big|\Big(\frac{\partial F_{i}^*}{\partial \lambda_{i}}\Big)^{-1} \sum_{j=1}^{i-1} \frac{\partial F_j^*}{\partial \lambda_{i}} \frac{\partial F_{i}^*}{\partial \lambda_j}\Big|  < |\sigma_{i-1}(J)|,\nonumber
\ees
and the proof is completed.
\end{proof}

This theorem provides the conditions for local convergence on each iteration-$i$. In particular, if the hypothesis in Theorem~\ref{thm2} is satisfied, we will find the solutions to \eqref{Fi} by iterating \eqref{iterateH} provided that we start with a sufficiently close initial condition. Notice also that this condition suggests that in practice the local convergence will be difficult to satisfy if the Jacobian matrix ${F}_{i-1,\bm{\lambda}_{i-1}}$ is close to singular. With these two theorems, we can now establish

\begin{theorem}
Let $\bm{\mu}^{(n)}\in\mathbb{R}^n$ be the solution of the n-dimensional system of equations in \eqref{Fn}. We assume the hypothesis in Theorem~\ref{thm2}, then the EBE method is locally convergent.
\end{theorem}

\begin{proof}
Choose an initial condition, $(\alpha_1,\ldots,\alpha_n)$, that is sufficiently close to the solution $\bm{\mu}^{(n)}$ of ${\bf F}_n(\bm{\lambda}_n)={\bf 0}$. First, let us define the surface $F_1(\lambda_1,\ldots,\lambda_n)=0$ as $\mathcal{M}_{n}$; here, the dimension of $\mathcal{M}_n$ is at most $n-1$. Subsequently, we define the surfaces
${\bf F}_2(\bm{\lambda}_n)=\bm{0}$ as $\mathcal{M}_{n-1}$,  ${\bf F}_3(\bm{\lambda}_n)=\bm{0}$ as $\mathcal{M}_{n-2}$, and so on. The dimension of $\mathcal{M}_{j}$ is at most $j-1$. We assume that ${\bf F}_n(\bm{\lambda}_n)=\bm{0}$ has at least a solution, then $\mathcal{M}_1$ contains the solution $\bm{\mu}^{(n)}$. It is clear that $\mathcal{M}_n \supset \mathcal{M}_{n-1} \supset \ldots \supset \mathcal{M}_1$.

For $i=1$, we solve $F_1(\lambda_1,\alpha_2,\ldots,\alpha_n)=0$ for $\lambda_1$. Geometrically, we look for the first coordinate on the surface $\mathcal{M}_n$. From the Assumption~1.2, we have the local convergence of the usual Newton's iteration. If $\alpha_1$ is sufficiently close to the solution $\bm{\mu}^{(1)}=\mu_1^{(1)}\in\mathbb{R}$, as $m\to\infty$ we obtain the solution $(\mu_1^{(1)},\alpha_2,\ldots,\alpha_n)\in \mathcal{M}_n$. By the smoothness assumption, $(\mu_1^{(1)},\alpha_2,\ldots,\alpha_n)$ is also close to $\bm{\mu}^{(n)}$.

Continuing with $i>1$, we want to solve ${\bf F}_i(\bm{\lambda}_i,\alpha_{i+1},\ldots,\alpha_n)={\bf 0}$ for $\bm{\lambda}_i$. Numerically,
we will apply the iterative map ${\bf H}_i$ in \eqref{iterateH} starting from $(\bm{\mu}^{(i-1)},\alpha_i,\ldots,\alpha_n)\in \mathcal{M}_{n-i+2}$.
By Assumption~1.2, the Jacobian, ${\bf F}_{i-1,\bm{\lambda}_{i-1}}(\bm{\mu}^{(i-1)},\alpha_i,\ldots,\alpha_n)$ is nonsingular so by implicit function theorem, for any local neighborhood $V$ of $\bm{\mu}^{(i-1)}$, there exists a neighborhood $U$ of $ \alpha_{i}$ and a $C^1$ function $\bm{h}_{i-1}:U\to V$ such that $\bm{\mu}^{(i-1)}=\bm{h}_{i-1}(\alpha_i)$ and ${\bf F}_{i-1}(\bm{h}_{i-1}(\lambda_i),\lambda_i,\alpha_{i+1},\ldots,\alpha_n) =0$ for all $\lambda_i\in U$. Since the initial condition $\alpha_i$ is close to $\mu_i^{(n)}$, by the smoothness assumption it is also close to $\mu_{i}^{(i)}$ that solves ${\bf F}_i(\bm{\lambda}_i,\alpha_{i+1},\ldots,\alpha_n)=0$. The continuity of $\bm{h}_{i-1}$ on $U$ means that $(\bm{\mu}^{(i)}_{i-1},\mu_i^{(i)}) \in V\times U$. Geometrically, this means the surface $F_i(\bm{\lambda}_i,\alpha_{i+1},\ldots,\alpha_{n})=0$ intersects with the curve $\bm{\lambda}_{i-1}=\bm{h}_{i-1}(\lambda_i)$ at $\bm{\mu}^{(i)} = (\bm{\mu}^{(i)}_{i-1},\mu_i^{(i)})$. Therefore, we can find the solution for this $i$-dimensional system by tracking along the curve $\bm{\lambda}_{i-1}=\bm{h}_{i-1}(\lambda_i)$ where we consider $\lambda_i$ as an independent parameter. The iterative map ${\bf H}_i$ in \eqref{mapH} is to facilitate this tracking and the conditions in Theorem~\ref{thm2} guarantee convergence to the solution. Notice that during this iteration, the solution remains on $\mathcal{M}_{n-i+2}$. The solution for this i-dimensional problem is $(\bm{\mu}^{(i)},\alpha_{i+1},\ldots,\alpha_n)\in\mathcal{M}_{n-(i+1)+2}\subset\mathcal{M}_{n-i+2}\subset \ldots\subset\mathcal{M}_{n}$. Continuing with the same argument, we find that for $i=n$, $\bm{\mu}^{(n)}\in \mathcal{M}_1\subset \mathcal{M}_n$.
\end{proof}

This iterative procedure finds the solution by searching along the manifold $\mathcal{M}_n$ in the direction of the hypersurfaces of a single parameter at a time, which local existence is guaranteed by the Assumption~1. It is clear that after each step-$i$, the estimated solution may not necessarily be closer to the true solution since the estimates do not minimize the closest path to the true solution along the manifold $\mathcal{M}_n$ (or the geodesic distance). This means that, locally,
\bes\|(\bm{\mu}^{(i+1)},\alpha_{i+2},\ldots,\alpha_n)-\bm{\mu}^{(n)}\|\leq \|(\bm{\mu}^{(i)},\alpha_{i+1},\ldots,\alpha_n)-\bm{\mu}^{(n)}\|\nonumber\ees for $i<n-1$ is not true.

In practice, when initial conditions are not closed to the solution, the (global) convergence of EBE requires the following additional condition: For every $i$, there exists a nonempty connected set that contains $(\bm{\mu}^{(i)},\alpha_{i+1})$ and $\bm{\mu}^{(i+1)}$ such that $\mathbf{F}_{i,\bm{\lambda}_{i}}$ evaluated at any point in this set is nonsingular. The existence of this set will allow us to build a path to connect these two points that are far apart. If this condition is not met, we need an additional treatment to overcome this issue which will be discussed in the next section.


\section{Practical challenges}\label{section5}
In this section, we will discuss several practical challenges related to our algorithm with remedies. They include non-locality of the initial condition, mistracking due to multiple solutions, non-existence of solutions within the desired numerical tolerance, and the computational complexity.

\subsection{Adaptive tracking}
As we mentioned in the previous section, the EBE method only converges locally, which means that it requires an adequate initial condition which is practically challenging. In our numerical simulations below, in fact, we always start from zero initial condition, $\alpha_i=0, \forall i=1,\ldots, n$. In this case, notice that even when we obtain an accurate solution at step-$i$, that is, ${\bf F}_{i}(\bm{\hat \mu}^{(i)})\approx {\bm 0}$, as we proceed to the next iteration, $|{F}_{i+1}(\bm{\hat\mu}^{(i)},\alpha_{i+1})|\gg 0$, meaning that $(\bm{\hat\mu}^{(i)},\alpha_{i+1})$ is not close to the solution, $\bm{\mu}^{(i+1)}$. Even when $\frac{\partial F_{i+1}}{\partial \lambda_{i+1}}(\bm{\hat\mu}^{(i)},\alpha_{i+1})$ is not singular, according to equation (\ref{scalarnewton}), $\lambda_i^{m+1}$ could be very far away from $\lambda_i^m$. In this case, Newton's method could fail in Eq. (\ref{corrector}) because the initial guess could be very far from the solution.

As a remedy, we employ an adaptive tracking on $\lambda_i$ to guarantee that the application of Newton's method is within its zone of convergence for each predictor-corrector step. The idea of the adaptive tracking is that we cut the tracking step, $\Delta \lambda_i := \lambda_{i+1}-\lambda_i$, by half until the prediction-correction step in (\ref{predictor})-(\ref{corrector}) converges. The detail algorithm is outlined below.

\begin{algorithm}[H]
\small\SetAlgoLined
 \SetAlgorithmName{Algorithm}{problem}{List of problems}
\SetKwInOut{Input}{Input}\SetKwInOut{Output}{Output} 
\Input{Minimum step size $\lambda_{min}$ and threshold value of $Tol$.}
Compute $\Delta \lambda_i$ by using Newton's method to solve $F_i=0$.\\
Set $Final=\Delta \lambda_i$\\
\While{$|Final|>0$}{
Solve $\mathbf{F_{i-1}}(\mathbf{\lambda_{i-1}},\lambda_i+\Delta \lambda_i)=0$ by using Newton's method;\\
\eIf{Newton's method fails}{$\Delta \lambda_i=\Delta \lambda_i/2$\\\If{$\Delta \lambda_i<\lambda_{min}$}{Discard the $i$-th equation}}{$Final=Final-\Delta \lambda_i$\\$\Delta\lambda_i=\min\{\Delta\lambda_i,Final\}$}
} \caption{Summary of adaptive tracking algorithm}
\end{algorithm}

\subsection{Bifurcation}
In order to solve $F_{i}(\lambda_1,\lambda_2,\cdots,\lambda_i)=0$, we track $\bm{F}_{i-1}(\bm{\lambda}_{i-1},\lambda_i)=\mathbf{0}$ along $\lambda_{i}$ as a parameter. During this parameter tracking, we may have some bifurcation points of $\lambda_{i}$ for the nonlinear system $\bm{F}_{i-1}(\bm{\lambda}_{i-1},\lambda_i)=\bm{0}$. This means that the Jacobian, $\bm{F}_{i-1,\bm{\lambda}_{i-1}}(\bm{\lambda}_{i-1},\lambda_i)$ is rank deficient such that $\bm{F}_{i-1}(\bm{\lambda}_{i-1},\lambda_i)=\bm{0}$ has multiple solutions $\bm{\lambda}_{i-1}$ for a given $\lambda_i$. In this situation, $F_i$ has multiple realizations functions of $\lambda_i$ (see the illustration in Figure~\ref{Fig:challenges} where the bifurcation point is the intersection of the two possible realizations of $F_i$). In this illustration, the goal is to track along the red branch to find the root, $F_i(\lambda_i)=0$. As we get closer to the bifurcation point, the Jacobian, $\bm{F}_{i-1,\bm{\lambda}_{i-1}}(\bm{\lambda}_{i-1},\lambda_i)$, is singular such that we can't evaluate (\ref{predictor}). Intuitively, the existence of multiple solutions near the bifurcation point induces a possibility of mistracking from the red curve to the green curve (as shown by the arrows) which prohibits one to find the solution.

\begin{figure}\centering
\includegraphics[width=.5\textwidth]{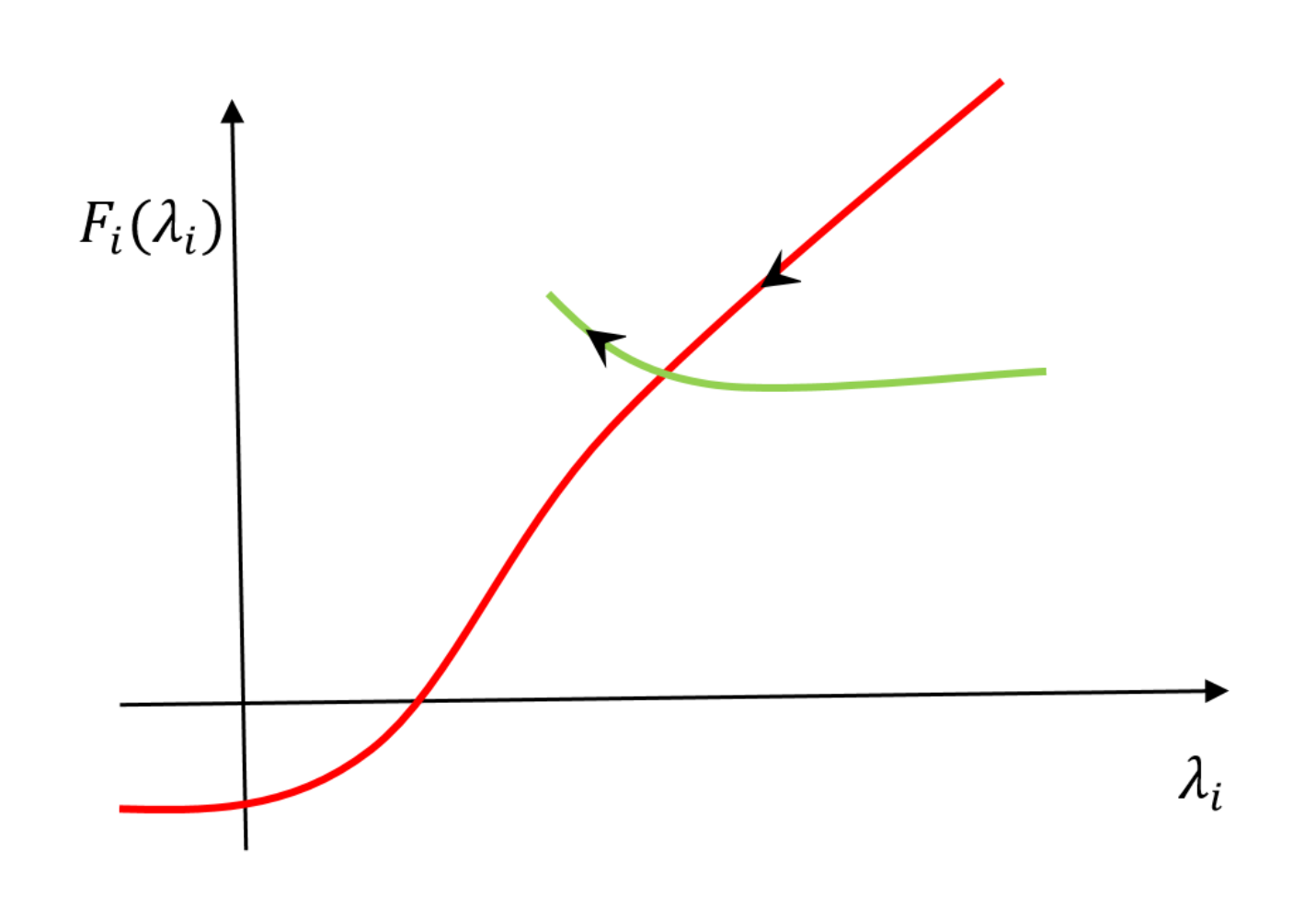}\includegraphics[width=.5\textwidth]{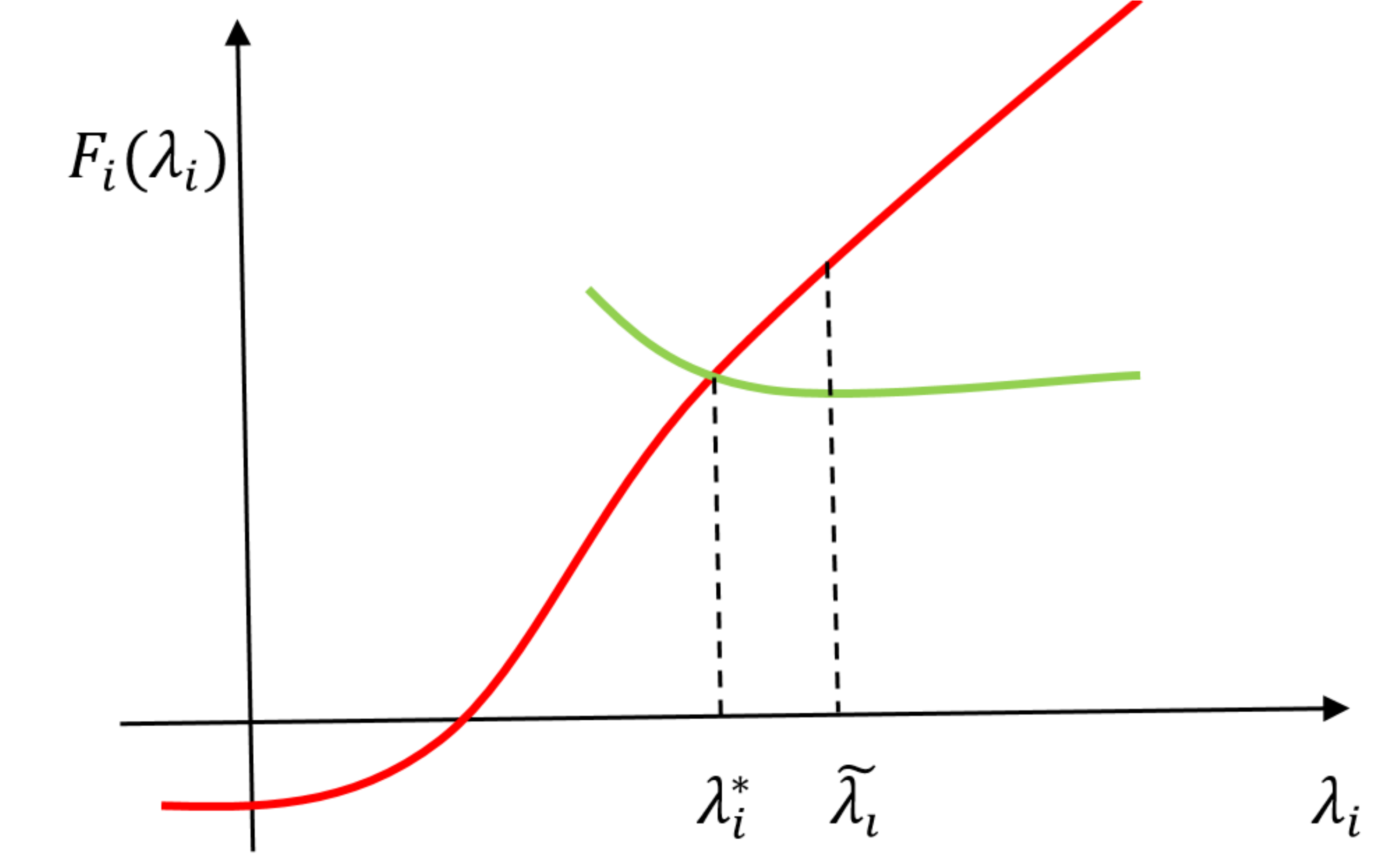}
\caption{Plot of $F_i(\lambda_i)$ v.s. $\lambda_{i}$: There are two bifurcation branches for the nonlinear system $\mathbf{F}_{i-1}(\mathbf{\lambda_{i-1}},\lambda_i)=\mathbf{0}$. The left part is a mistracking example; the right part is the illustration of a numerical method to avoid the bifurcation point.} 
\label{Fig:challenges}
\end{figure}

To avoid such mistracking, we apply the deflation technique to compute the bifurcation point directly \cite{HHHLSZ2,LVZ}. Once the bifurcation point is estimated, we approximate the correct branches using the Richardson extrapolation to avoid mistracking. Denote the bifurcation point as $\lambda_i^*$, the nonlinear system $\bm{F}_{i-1}(\bm{\lambda}_{i-1},\lambda_i)=\bm{0}$ is difficult to solve when $\lambda_{i}$ is close to $\lambda_i^*$ since the Jacobian of $\bm{F}_{i-1}(\bm{\lambda}_{i-1},\lambda_i)$ becomes near singular. If the last attempt is $(\tilde{\bm{\lambda}}_{i-1},\tilde{\lambda}_i)$, we compute $(\bm{\lambda}_{i-1}^*,\lambda_i^*)$ by solving the following deflated system:
$$G(\bm{\lambda}_{i-1}^*,\lambda_i^*,\bm{v})=\left[\begin{array}{c}
\bm{F}_{i-1}(\bm{\lambda}_{i-1},\lambda_i)\\
\bm{F}_{i-1,\bm{\lambda}_{i-1}}(\bm{\lambda}_{i-1},\lambda_i)\bm{v}\\
\bm{\xi}^T\bm{v}-1
   \end{array}
   \right]=\bm{0},
$$ where $\bm{v}$ is the kernel of $\bm{F}_{i-1,\bm{\lambda}_{i-1}}(\bm{\lambda}_{i-1},\lambda_i)$  and $\bm{\xi}$ is a random vector to guarantee that $\bm{v}$ is not a zero eigenvector. In this case, $G(\bm{\lambda}_{i-1}^*,\lambda_i^*,\bm{v})$ is well-conditioned \cite{HHHLSZ2,LVZ}. Once the bifurcation point $(\bm{\lambda}_{i-1}^*,\lambda_i^*)$ is estimated, we can avoid mistracking by setting $\lambda_i=2\lambda_i^*-\tilde{\lambda}_i$ and solve  $\bm{F}_{i-1}(\bm{\lambda}_{i-1},\lambda_i)=\bm{0}$ by using Newton's method with an initial guess $2\bm{\lambda}_{i-1}^*-\tilde{\bm{\lambda}}_{i-1}$ (which is a Richardson extrapolation).

\subsection{Nonexistence of solutions}
In general, the moment constrained maximum entropy problems may not necessarily have solutions. Even when the  solutions exist theoretically, they could be difficult to find numerically due to the noisy dataset, error in the numerical integration, etc. In this case, we simply discard the equation $F_i$ when the minimum is larger than the desired tolerance. This feature (discarding the constraints that give no solutions) is only feasible in the EBE algorithm. However, some theories are needed to preserve the convexity of the polynomials in the exponential term of Eq. (\ref{expoly}) while discarding some of these constraints. In our numerical simulations below, we handle this issue by re-ordering the constraints. In particular, for a problem with moment constraints up to order-$4$, we include the constraints corresponding to $\mathbb{E}[x_i^4]$ $(i=1,\cdots,d)$ in the earlier step of the EBE iterations to avoid these constraints being discarded. Note that this method is sensitive to ordering, that is, different ordering of constraints  yields different path to compute the solution. Therefore, a systematic ordering technique that simultaneously preserves the convexity of the polynomial in the exponential term of Eq. (\ref{expoly}) is an important problem to be addressed in the future.

\subsection{Computational complexity}
The most expensive computational part in EBE is the numerical evaluation of (\ref{Poly}). For a fast numerical integration, we store the monomial basis $c_{\bm{j}}(\bm{x})$ as a matrix of size $N_\ell\times n$, where $N_\ell$ is the number of sparse grid points and $n$ is number of monomial basis. In this case, the computational cost in evaluating $F_j$ is $(2j+1)N_\ell$ ($j-1$ additions, $j+1$ multiplications and 1 subtraction for each grid point), excluding the computational cost for exponential function evaluation, which is on the order of $\log^2m$ to obtain an error of resolution $2^{-m}$ \cite{ahrendt:99}.
For the $i$-th iteration of the EBE algorithm, the computational cost to evaluate the $i$-dimensional system $\bm{F}_i$ is
 $\sum_{j=1}^i(2j+1)N_\ell=\frac{i^2+i}{2}N_\ell$, excluding the exponentiation.

\section{Numerical results}\label{section6}

In this section, we show numerical results of the EBE method on five examples. In all of the simulations below, unless stated, we set the Newton's tolerance $Tol_1={10^{-1}}$ and the predictor tolerance $Tol_2={10^{-10}}$. 
In the first test example, we will describe how the EBE method works on each iteration. The goal of the second example is to demonstrate the global convergence with solutions that are far away from initial condition, $\alpha_j=0$. In particular, we will test the EBE method on a problem with solutions, $\lambda_j$, that have magnitudes ranging from order $10^0-10^3$. In this example, we will show the robustness of the estimate as a function of the number of integration points (or the sparse grid level $\ell$). The third example is to demonstrate the performance on high dimensional problems (with $70\leq n\leq 310$ of order hundreds), induced from order-four moments of four to seven dimensional density functions. While these first three examples involve estimating densities of the form \eqref{expoly}, in the next two examples, we also test the EBE method to estimate densities from a given data set where the maximum entropy solutions may or may not exist. {\color{black} The first data-driven problem is to estimate densities of the first two leading EOFS of the wind stress-driven large-scale oceanic model \cite{abramov2009,abramov2010}. The second data-driven problem is to estimate two- to five-dimensional densities arising from solutions of the Kuramoto-Sivashinsky equation. In these two problems, we compare our method with the classical Newton's method, the MATLAB built-in solver {\asciifamily fsolve.m}, and the previously developed BFGS-based method \cite{abramov2009,abramov2010}.}


\begin{examp}\label{Ex1}
We consider a simple example $\rho(x)\propto \exp(x+x^2+x^3)$ for $x\in[-1,1]$ so that the exact solution is ${\bm{\lambda}}=(1, 1, 1)$. Here, the moments $f_j$ can be computed numerically as follows,
\[f_j=\frac{\int_{-1}^1 x^j \rho(x)dx}{\int_{-1}^1 \rho(x)dx},\quad \hbox{~for~} i=1,2,3.\] In order to numerically integrate both the denominator and numerator, we used a regular one-dimensional sparse grid of level $\ell=7$ (the number of nodes is 65). Our goal here is to illustrate the method and to show the trajectory of the solutions after each iteration of the inner loop $m$ and outer loop $i$.
In Figure~\ref{Fig:exp}, we show the surface of $F_1(\lambda_1,\lambda_2,\lambda_3)=0$ (grey). For $i=1$, we solve the $F_1(\lambda_1,0,0)=0$, after three iterations ($m=3$) the solution converges to $\lambda_1=2.3$ (see Table~\ref{Tab:iter}). For $i=2$, we start with this solution and introduce the second variable $\lambda_2$ for solving the second equation $F_2(\lambda_1,\lambda_2,0)=0$ with constraint $F_1(\lambda_1,\lambda_2,0)=0$. Here, the solution follows the path $\lambda_1=h_1(\lambda_2)$ thanks to the implicit function theorem (black curve). Numerically, a sequence of (green) points following this path converges to a point that satisfies $F_1(\lambda_1,\lambda_2,0)=F_2(\lambda_1,\lambda_2,0)=0$ (the green point in the intersection between black and red curves in Figure~\ref{Fig:exp}). In the next iteration $i=3$, we introduce the third variable $\lambda_3$ for solving the third equation $F_3(\lambda_1,\lambda_2,\lambda_3)=0$ with constraints $F_1(\lambda_1,\lambda_2,\lambda_3)=F_2(\lambda_1,\lambda_2,\lambda_3)=0$. By the implicit function theorem, we have $(\lambda_1,\lambda_2)=h_2(\lambda_3)$  that satisfies $F_1(h_2(\lambda_3),\lambda_3)=F_2(h_2(\lambda_3),\lambda_3)=0$, which is shown in red curve in Figure~\ref{Fig:exp}. On this red curve, we have a sequence of (blue) points which converges to the solution of the full system (cyan point shown in Figure~\ref{Fig:exp}). The coordinate of the solution on each iteration is shown in Table~\ref{Tab:iter}. Notice that the solutions always lie on the surface $F_1(\lambda_1,\lambda_2,\lambda_3)=0$.

\begin{figure}\centering
\includegraphics[width=4in]{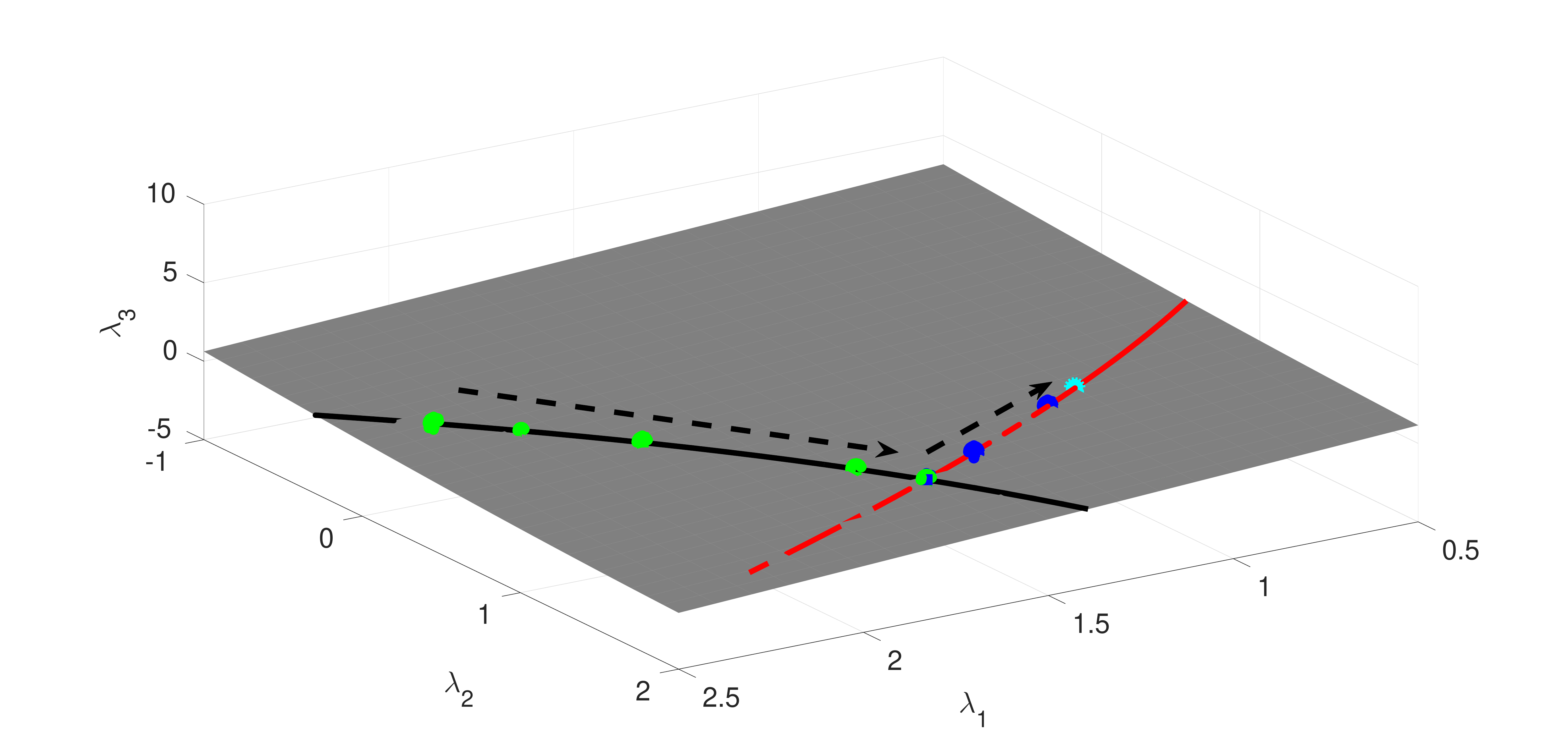}
\caption{The illustration of Example~\ref{Ex1}: The black curve is $\lambda_1=h_1(\lambda_2)$, the green points are the iterations when we solved $F_1(\lambda_1,\lambda_2,0)=0$; The red curve is $(\lambda_1,\lambda_2)=h_2(\lambda_3)$, the blue points are the iterations when we solved $F_1(\lambda_1,\lambda_2,\lambda_3)=F_2(\lambda_1,\lambda_2,\lambda_3)=0$; The cyan point is the numerical solution.}\label{Fig:exp}
\end{figure}

\begin{table}
\renewcommand{\arraystretch}{1.8}
\begin{center}
\caption{The coordinate of the solutions of Example~\ref{Ex1} for each iteration, starting from $(0,0,0)$. For each outer loop $i$, the EBE takes few iterates ($m$) to find the $i-$dimensional solution, fixing $\lambda_{j}=\alpha_j=0$ for $j>i$.} \label{Tab:iter}
\begin{tabular}{|c|c|c|c|}\hline
\backslashbox{$m$}{$i$}&1&2&3\\\hline
0 &(0,0,0)&{\color{green}{(2.30,0,0)}}&{\color{blue}{(1.58,1.43,0)}}\\\hline
1&(1.76,0,0)&{\color{green}{(2.23,0.22,0)}}&{\color{blue}{(1.52,1.38,0.26)}}\\\hline
2&(2.23,0,0)&{\color{green}{(1.87,0.57,0)}}&{\color{blue}{(1.12,1.09,0.76)}}\\\hline
3&(2.30,0,0)&{\color{green}{(1.67,1.21,0)}}&{\color{cyan}{(1,1,1)}}\\\hline
4&&{\color{green}{(1.58,1.43,0)}}&\\\hline
\end{tabular}
\end{center}
\end{table}

\end{examp}

\begin{examp}\label{Ex2}
We consider a one-dimensional example with up to order-six moment constraints with explicit solution given by, \[\rho(x)\propto\exp\Big(
2x+16x^2+24x^3+96x^4 -256x^5 -1024x^6\Big),\]
as shown in Figure~\ref{Fig:ex2}. This example is a tough test problem since the solution, $\bm{\lambda}=(2,16,24,96,-256,1024)$, has components of order $10^0-10^3$. Similar to Example~\ref{Ex1}, we compute the moments $f_i$ by using a one-dimensional sparse grid of level $\ell=7$ (65 nodes).   The EBE algorithm converges to the exact solution with error, $\|\bm{\lambda}-\bm{\lambda}^*\|=5.44\times10^{-13}$. Since the numerical experiment is performed with an initial condition $\alpha_j=0$ that is far from the solution, this result demonstrates a global convergence of the EBE method.

Next, we investigate the sensitivity of the estimates to the number of sparse grid points used in approximating the integral. In our numerical experiments, we estimate the true moments $f_i$ using one-dimensional sparse grid of level $\ell=20$ (524,289 nodes) and feed these moment estimates into the EBE algorithm. In Figure~\ref{Fig:tol}, we show the error in $\lambda$ (with $\ell_2$ metric) for different levels of the sparse grid from 6 to 15 that are used in the EBE method. Notice that the error decreases as a function of $\mathcal{\ell}$ and the improvement becomes negligible for $\ell>8$.

\begin{figure}\centering
\includegraphics[width=.7\textwidth]{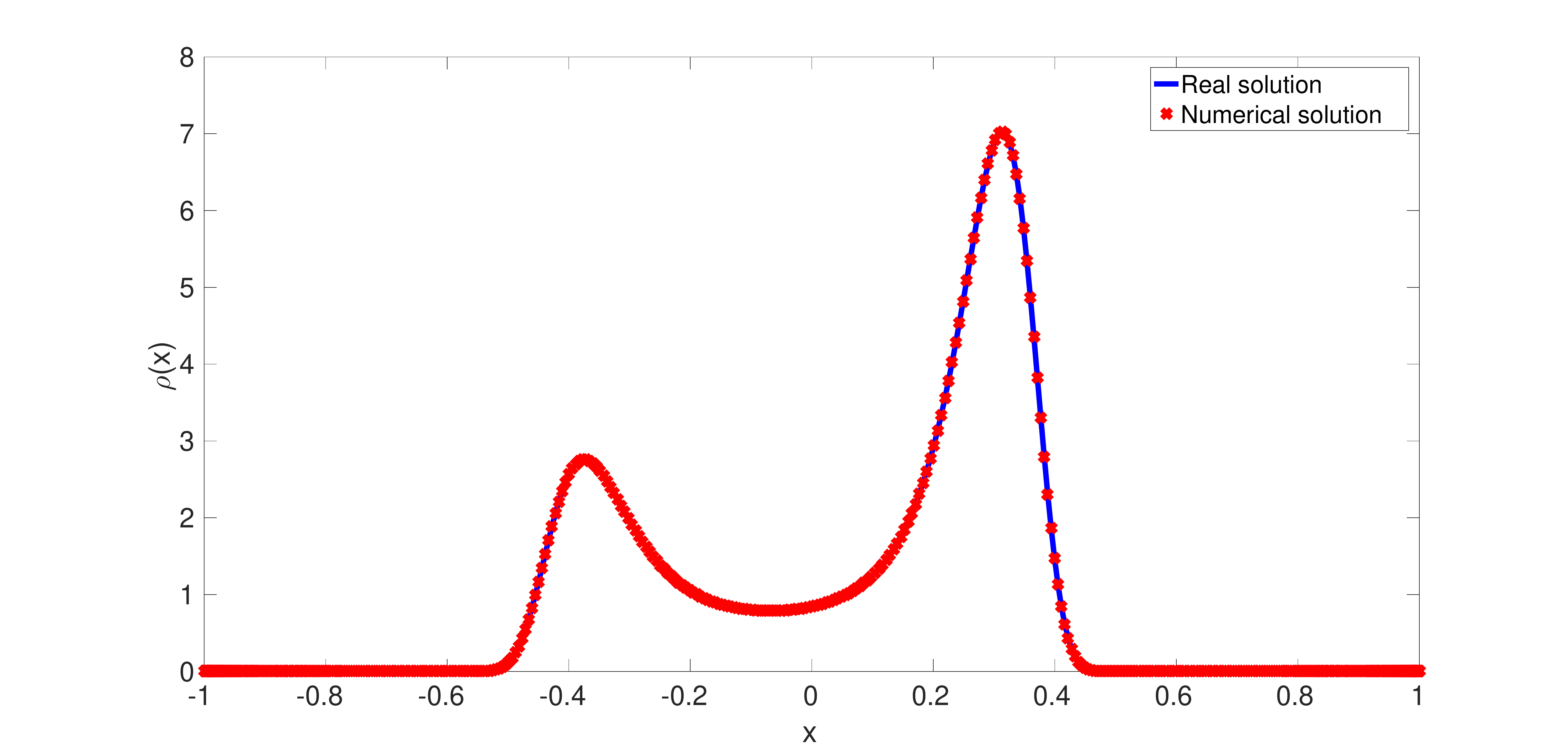}
        \caption{The unnormalized density $\rho(x)$ in Example 2}\label{Fig:ex2}
\end{figure}

\begin{figure}\centering
\includegraphics[width=.7\textwidth]{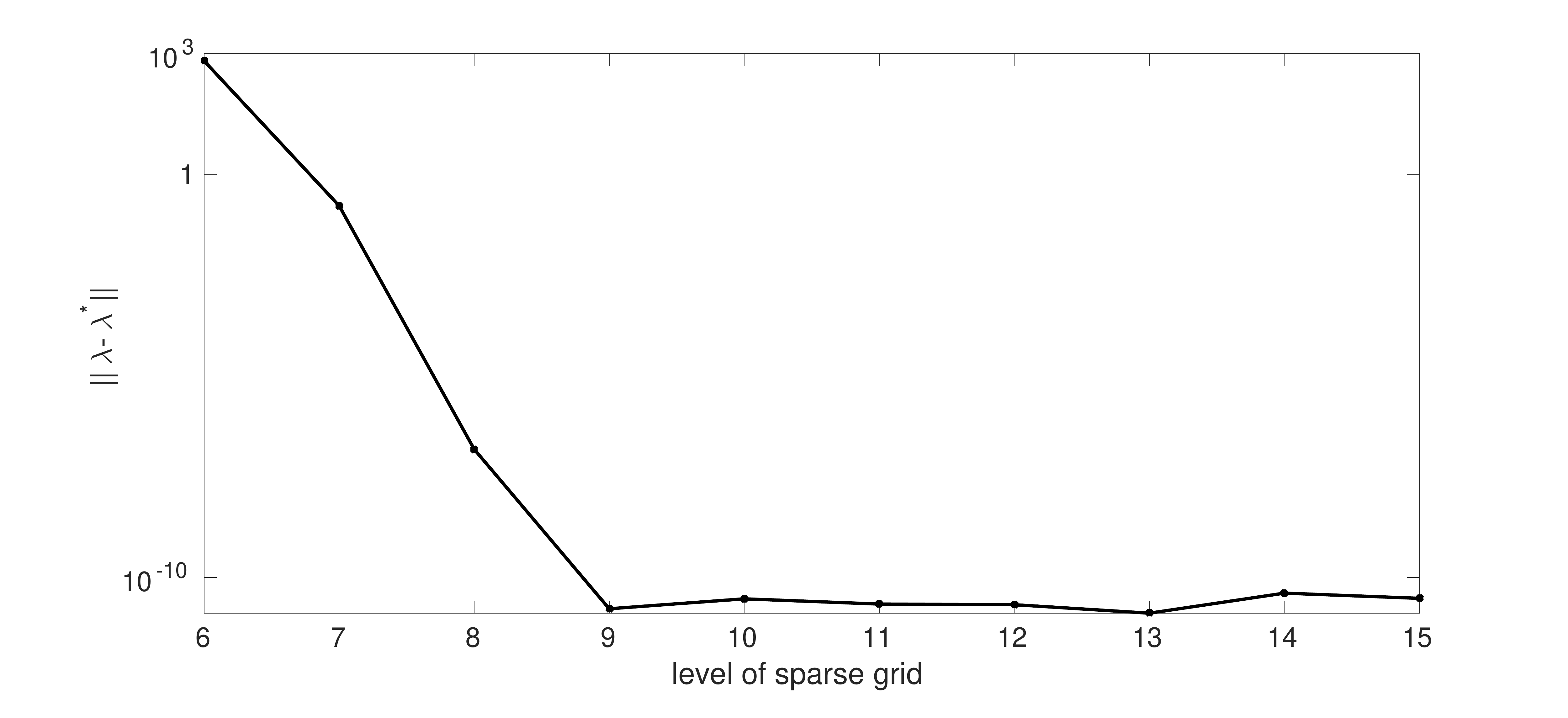}
        \caption{The solution error as a function of the number of sparse grid.}\label{Fig:tol}
\end{figure}

\end{examp}

\begin{examp}\label{Ex3}
In this example, we consider a $d$-dimensional example with
an explicit solution,
\[\rho(x)\propto \exp(-2x_1^4+x_2^3-x_2^4-x_3^4-1.8x_4^4),\]
on domain $\Omega=[-1,1]^d$ where we will vary $d=4,\ldots, 7$. For these simulations, we consider up to order-four moment constraints and fix the sparse grid level $\ell=8$ to compute the integration. 

Here, the EBE method is able to estimate $\bm{\lambda}$ with $\ell_2$-errors of order $10^{-13}$ (the error in $\bm{\lambda}$ is $1.11\times10^{-13}$ and moments error is $3.15\times10^{-15}$). 
In this computation, the dimensions of the nonlinear system are $70$ for $d=4$, $126$ for $d=5$, $210$ for $d=6$, and $310$ for $d=7$.
Here, the EBE method is able to recover the true density even if we prescribe more constraints, corresponding to $d$ larger than four.


\end{examp}

\begin{examp}\label{Ex4}

Next, we consider estimating a two-dimensional probability density of the two leading empirical orthogonal functions of a geophysical model for wind stress-driven large-scale oceanic model \cite{mccalpin:95,mh:96}. This is exactly the same test example in the previously developed BFGS-based method \cite{abramov2009,abramov2010}. In fact, the two-dimensional density that we used here was supplied by Rafail Abramov. First, we compare the EBE method with the BFGS algorithm of \cite{abramov2009} which code can be downloaded from \cite{maxentcode}. In this comparison, we use the same uniformly distributed grid points where the total number of nodes are {$85\times85=7,225$}. We set the Newton's tolerance of the EBE algorithm to be $10^{-10}$. In Table~\ref{Tab:comp} notice that the moment errors of the EBE are much smaller compared to those of the BFGS method.

While the EBE is superior compared to BFGS, we should note that the BFGS method does not use the Hessian of ${\bf F}_i$ whereas the EBE does. For a fair comparison, we include results using the MATLAB built-in function {\asciifamily fsolve.m}, which default algorithm is the trust-region-dogleg (see the documentation for detail \cite{fsolve}). In our numerical implementation, we apply {\asciifamily fsolve.m} with a specified Hessian function ${\bf F}_n$. We also include the classical Newton's method with a specified Hessian function ${\bf F}_n$. In this comparison, we use the same sparse grid of level $\ell=11$ (or 7,169 nodes) to compute the two-dimensional integral. Notice that the EBE method is still superior compared to these two schemes as reported in Table~\ref{Tab:comp}. In fact, Newton's method does not converge for higher-order moment constraints. The joint two-dimensional PDFs are shown in Figure~\ref{Fig:comp}. The first row is the two-dimensional density function provided by R. Abramov. The second row shows the EBE estimates using up to order four-, six-, and eight-moment constraints. The third and fourth rows show the BFGS and MATLAB {\asciifamily fsolve.m} estimates, respectively.

\begin{table}
\renewcommand{\arraystretch}{1.8}
\begin{center}
 \caption{Summary of solutions for Example~\ref{Ex4}: Moment errors for different algorithms with different grids.}\label{Tab:comp}
\begin{tabular}{|c|c|c|c|}\hline
Methods&\multicolumn{3}{|c|} {order}\\\cline{2-4}&4&6&8\\\hline
BFGS algorithm with uniform grid&$4.07\times10^{-2}$&$1.45\times10^{-4}$&$1.14\times10^{-2}$\\\hline
EBE algorithm with uniform grid&$1.27\times10^{-11}$&$9.84\times10^{-15}$&$7.75\times10^{-13}$\\\hline
EBE algorithm with sparse grid&$7.54\times10^{-12}$&$8.12\times10^{-15}$&$2.43\times10^{-13}$\\\hline
Matlab {\asciifamily fsolve.m} with sparse grid&$4.70\times10^{-7}$&$1.19\times10^{-4}$&$1.74\times10^{-4}$\\\hline
Newton with sparse grid&$5.12\times10^{-11}$&diverge&diverge\\\hline
\end{tabular}
\end{center}
\end{table}

\begin{figure}\centering
\includegraphics[width=.4\textwidth]{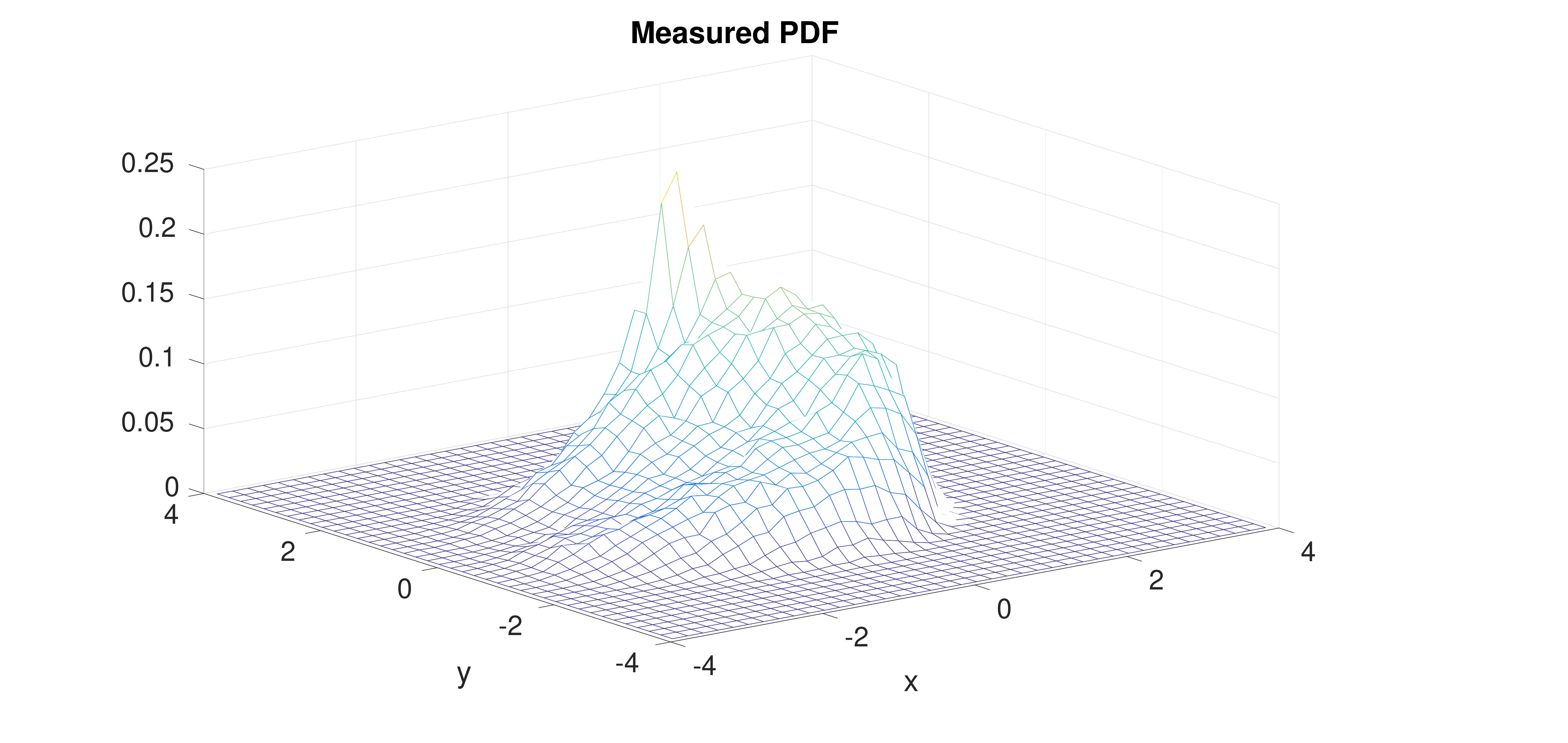}
\includegraphics[width=\textwidth]{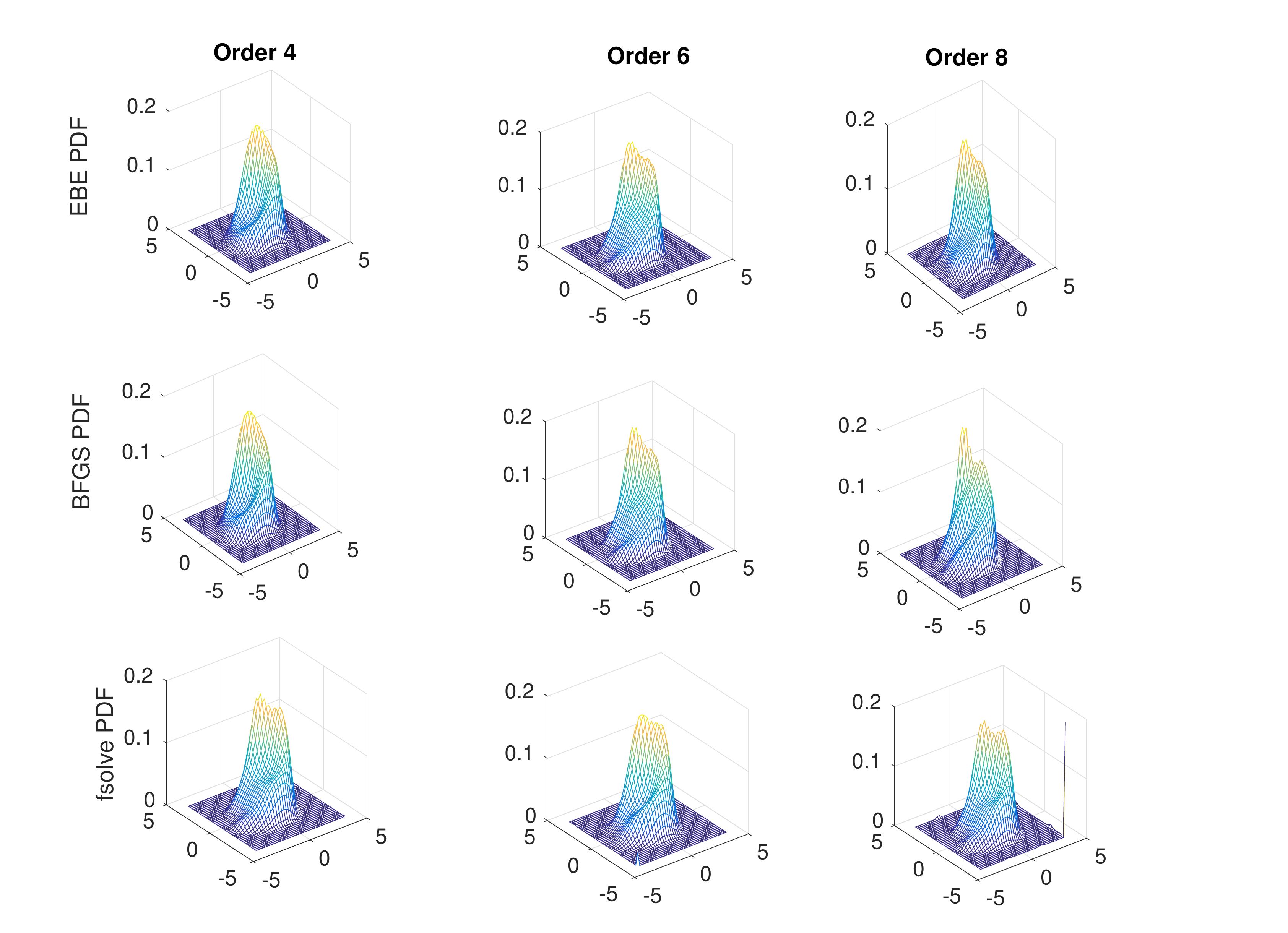}
        \caption{The 2D measured probability density functions supplied by R. Abramov (first row); PDFs computed by the EBE method (second row), BFGS algorithm (third row), and the MATLAB {\asciifamily fsolve.m} function (fourth row).}\label{Fig:comp}
\end{figure}

\end{examp}

\begin{examp}\label{Ex5} In this example, we consider estimating multidimensional densities of the solutions of the Kuramoto-Sivashinsky equation. Here, the solutions are integrated with a fourth-order time-differencing method on $128$ equally spaced grid points over a domain of $[0,32\pi]$ as in \cite{kt:05}. We use initial condition $u(x, 0) = \cos(x/(16\xi))(1 + \sin(x/16))$, where $\xi\sim U[0,1]$ and integration time step of $0.25$. The data is generated by integrating 10,000 time steps. Based on this data set, we randomly select $d$ components and estimate the $d$-dimensional joint density associated to these components. For visual comparison, we also show the results from a two-dimensional kernel density estimation method \cite{rosenblatt1956,parzen1962} as a reference. Numerically, we use the MATLAB built-in function, {\asciifamily ksdensity.m}.
Note that the BFGS algorithm \cite{abramov2009} does not work on this data set while the classical Newton's method only converges for the two-dimensional case. We also show the corresponding results with the MATLAB {\asciifamily fsolve.m} with specified Hessian function as in the previous example. The moment errors of these three schemes are reported in Table~\ref{Tab:KS}.

In Figure~\ref{Fig:KS2d}, we show the two-dimensional density estimated by EBE algorithm compared to those from the {\asciifamily fsolve.m}, the classical Newton's method, and the 2D kernel density estimate.
For the two-dimensional case, the resulting densities are visually identical although the corresponding moment error of the EBE method is still the smallest compared to the Newton's and the MATLAB {\asciifamily fsolve.m} (see Table~\ref{Tab:KS}). In Figure~\ref{Fig:KS3d}, we show the contour plot of the two-dimensional marginal densities obtained from solving the three-dimensional problem given four-moment constraints with the EBE method and the MATLAB {\asciifamily fsolve.m}. For diagnostic purpose, we also provide the corresponding contour plots of the two-dimensional kernel density estimates. Notice that the MATLAB {\asciifamily fsolve.m} produces completely inaccurate estimate. The EBE method produces an estimate that qualitatively agrees to the corresponding two-dimensional KDE estimates. The slight disagreement between these estimates are expected since we only provide up to order-four moments information.

In Figure~\ref{Fig:KS4d}, we show the results for the four-dimensional problem. We do not show the estimate from the MATLAB {\asciifamily fsolve.m} since it is not accurate at all. Here, we include more than four-order moments. Specifically, the total number of constraints for up to order-four moments is 70 while this result is based on 87 constraints, including 17 additional higher-order moment constraints that include order-six moments, $\mathbb{E}[x_i^6], i=1,\ldots, 4$. See the movie of the density estimates for each iteration in the supplementary material \cite{ebemovie}. Notice that the marginal densities estimated by the EBE look very similar to those estimated by the two-dimensional kernel density estimation. If more constraints are included, we found that we lose the convexity of the polynomial terms in \eqref{expoly}. As we mentioned before, we need a better criteria to preserve the convexity of the solutions.

In Figure~\ref{Fig:KS5d}, we include the result from a five-dimensional simulation. We also do not show the estimate from the MATLAB {\asciifamily fsolve.m} since it is not accurate at all. In this five-dimensional case, the EBE method automatically discards 34 equations (moment constraints). 
In this case, we suspect that either the maximum entropy solution that accounts for all of the constraints does not exist or the EBE method cannot find the solution. Here, the EBE method just estimates the best fitted solution within the tolerance of $10^{-10}$ by solving 91 out of 125 moment constraints. 

\begin{table}
\renewcommand{\arraystretch}{1.8}
\begin{center}
 \caption{Summary of solutions for Example~\ref{Ex5}.}\label{Tab:KS}
\begin{tabular}{|c|c|c|c|}\hline
d&EBE method& fsolve & Newton\\\hline
2& $1.098\times10^{-15}$&{$9.779\times10^{-7}$}&{$8.128\times10^{-14}$}\\\hline
3& $4.29\times10^{-13}$&{$3.150\times10^{-2}$}&{diverge}\\\hline
4& $1.19\times10^{-14}$&{$0.021$}&{diverge}\\\hline
5& $2.47\times10^{-11}$&{$0.018$}&{diverge}\\\hline
\end{tabular}
\end{center}
\end{table}

\begin{figure}\centering
\includegraphics[width=0.9\textwidth]{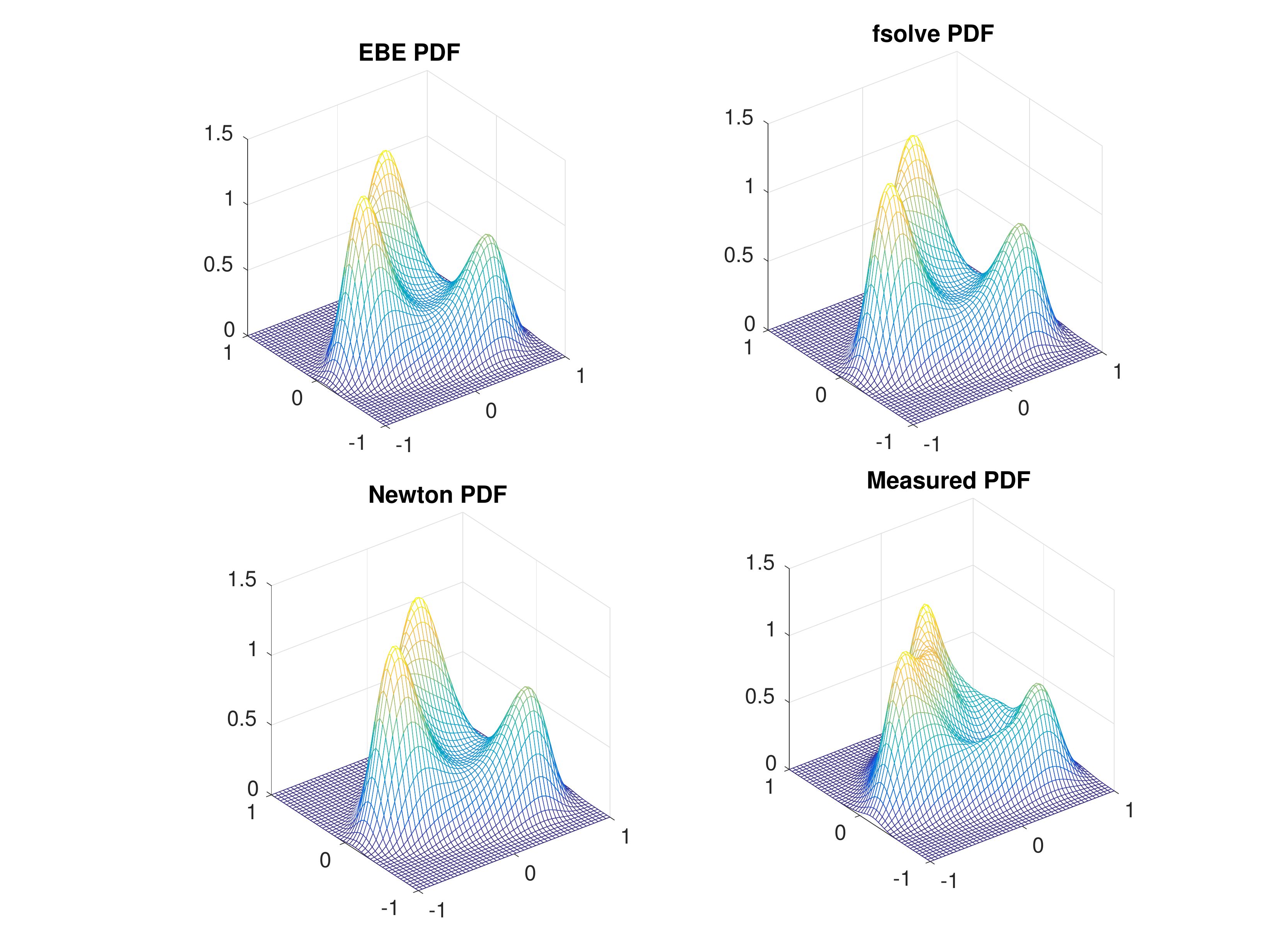}
        \caption{The comparison of the density functions obtained by the EBE algorithm, the MATLAB {\asciifamily fsolve.m} function, Newton's method, and the kernel density estimate (denoted as the measured pdf) for the two-dimensional case.}\label{Fig:KS2d}
\end{figure}
\begin{figure}\centering
\includegraphics[width=.9\textwidth]{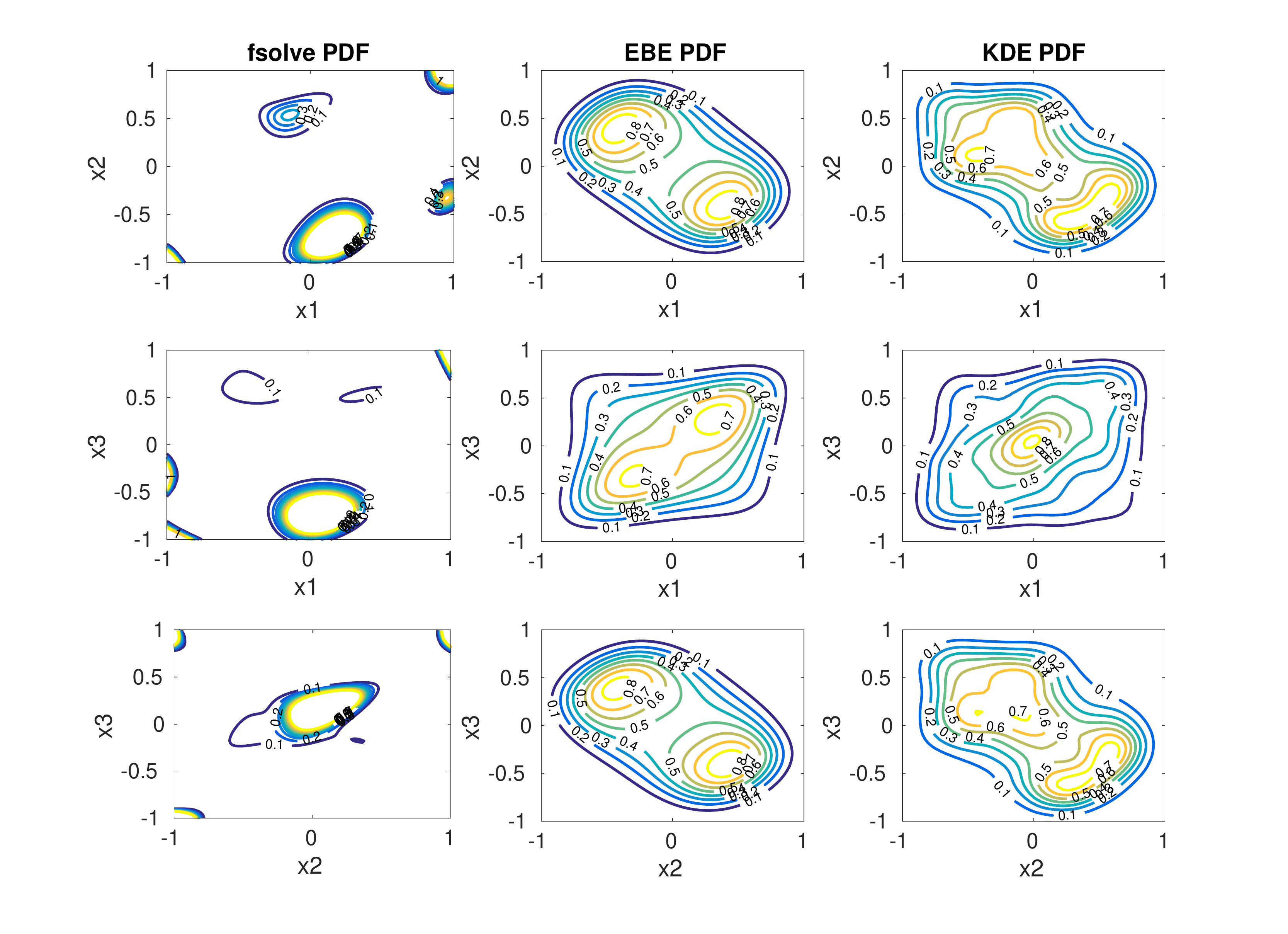}
        \caption{The comparison of the two-dimensional marginal density functions obtained by the MATLAB {\asciifamily fsolve.m} function (first column), the EBE algorithm (second column) that solves a three-dimensional problem accounting up to order-four moment constraints, and the two-dimensional kernel density estimate (third column).}\label{Fig:KS3d}
\end{figure}

\begin{figure}\centering
\includegraphics[width=0.8\textwidth]{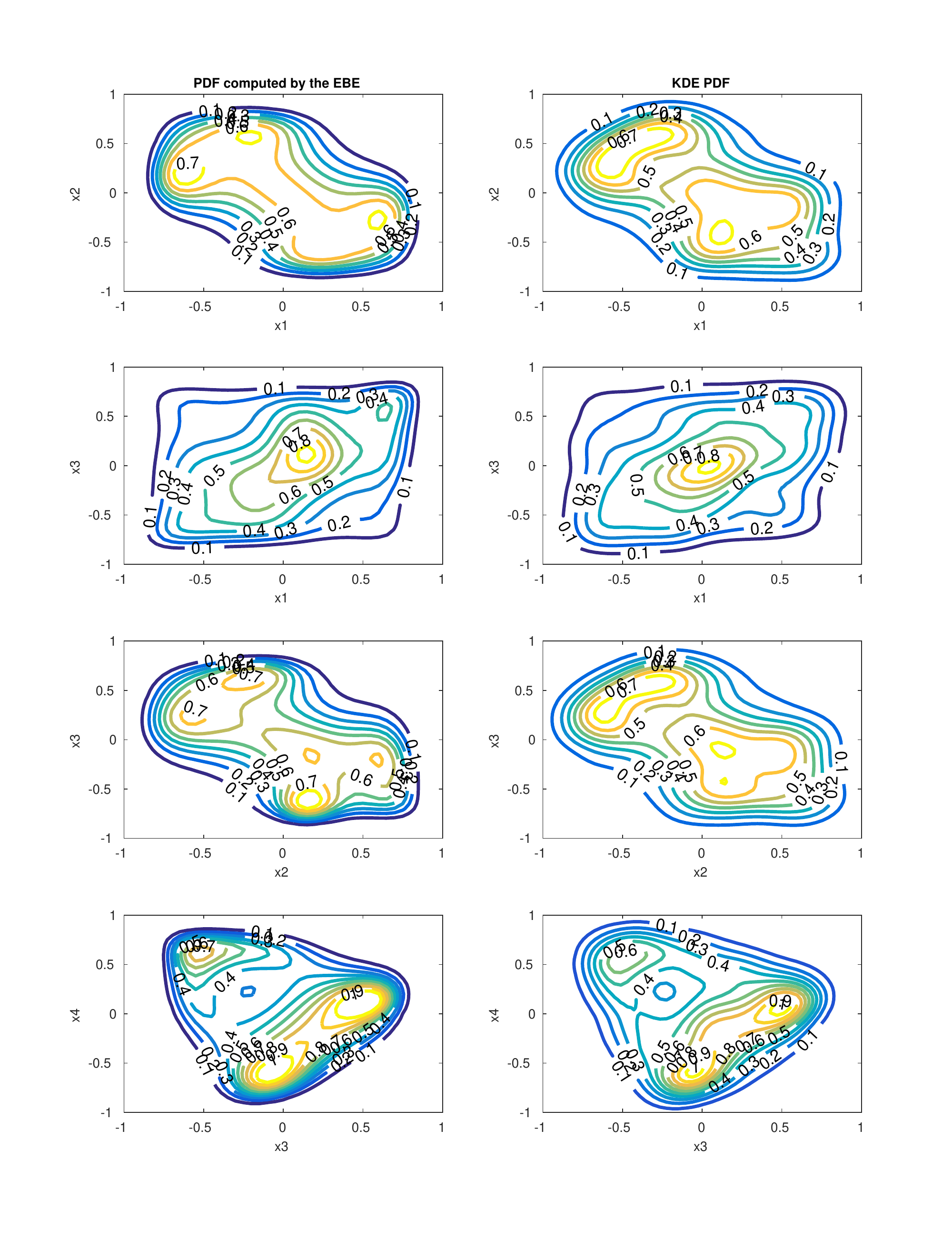}
        \caption{The comparison of the two-dimensional marginal density functions obtained by the EBE algorithm (first column) that solves a four-dimensional problem accounting {\color{black}more than} order-four moment constraints (see text for detail) and the two-dimensional kernel density estimate (second column).}\label{Fig:KS4d}
\end{figure}

\begin{figure}\centering
\includegraphics[width=4in,height=4in]{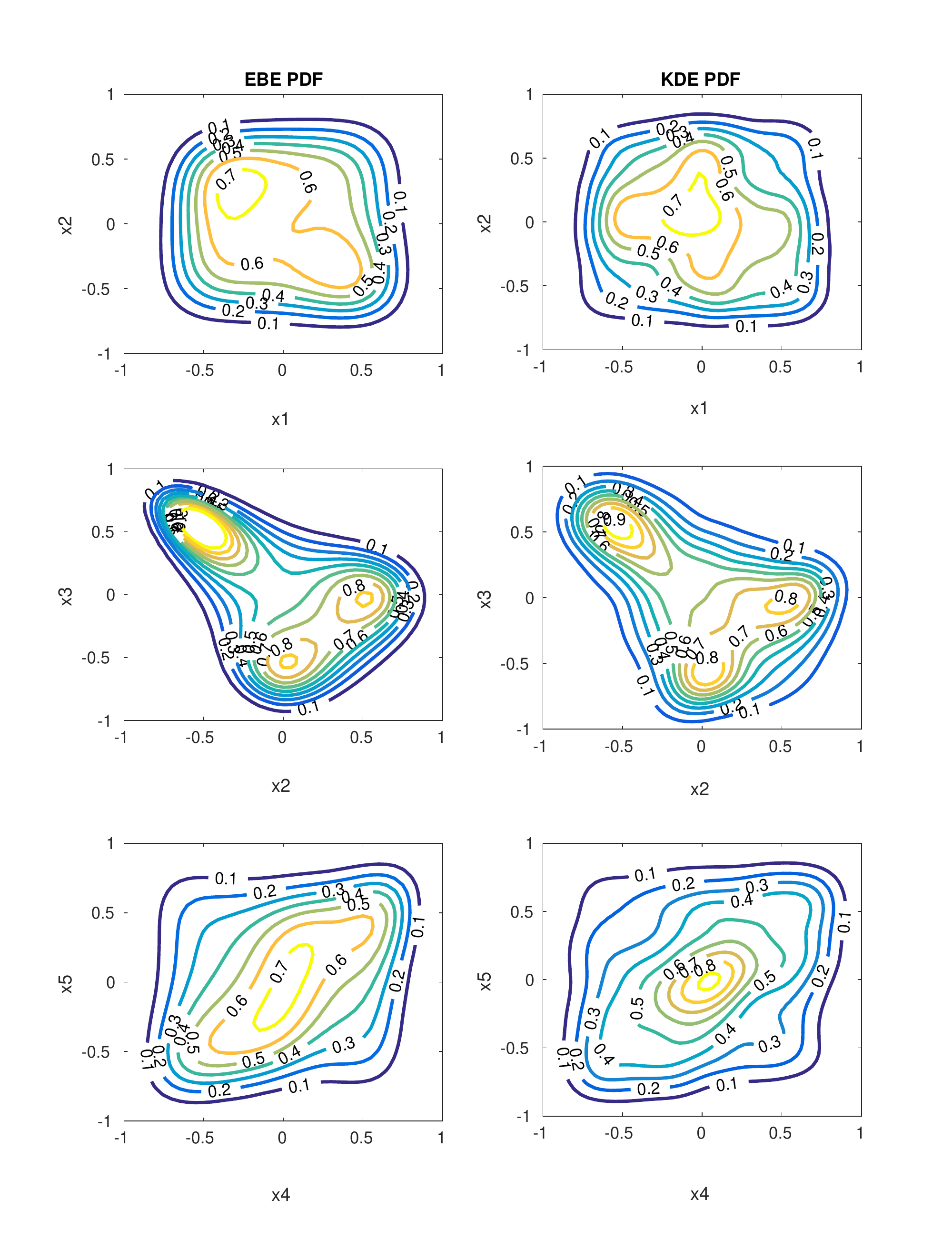}
        \caption{The comparison of the two-dimensional marginal density functions obtained by the EBE algorithm (first column) that solves a five-dimensional problem {\color{black}accounting the automatically selected, 91 out of the prescribed 125 moments,} and the two-dimensional kernel density estimate (second column).}\label{Fig:KS5d}
\end{figure}

\end{examp}

\section{Summary}\label{section7}

In this paper, we introduced a novel equation-by-equation algorithm for solving a system of nonlinear equations arising from the moment constrained maximum entropy problem. Theoretically, we have established the local convergence and provided a sufficient condition for global convergence. Through the convergence analysis, we understood that the method, geometrically, finds the solution by searching along the surface corresponding to one component of the nonlinear equations. Numerically, we have demonstrated its accuracy and efficiency on various examples. In one of the examples, we found that the EBE algorithm produces more accurate solutions compared to the previously developed BFGS-based algorithm which does not use the Hessian information \cite{abramov2009,abramov2010}. In this same example, we also found that the EBE is superior compared to two schemes that use the Hessian information, including the current MATLAB built-in solver which uses the trust-region-dogleg algorithm and the classical Newton's method.

We also found that the proposed EBE algorithm is able to solve a system of 70-310 equations when the maximum entropy solution exists compared to the previously developed BFGS method which was shown to work for a system of size 44-83 equations. On the Kuramoto-Shivashinski example, the EBE method is able to reconstruct the density of a four-dimensional problem accounting up to order-four moments (or 70 constraints). In this case, we showed that the estimate is improved by accounting for 17 additional constraints of order-six moments. For the five-dimensional problem with moments up to order-four, the EBE method reconstructs the solution within the desired precision, $10^{-10}$, by automatically selecting a subset of 91 constraints from the total prescribed 125 constraints induced by moments of up to order-four.

While the automatic constraint selection is a desirable feature since the maximum entropy solutions within the tolerance may not be easily estimated (nor theoretically available), further study is required to fully take advantage of this feature. In particular, an important open problem is to develop a mathematical theory for ordering the constraints since the path of the solution is sensitive to the order of the constraints. Simultaneously, the ordering of the constraints need to preserve the convexity of the polynomials in the exponential term of \eqref{expoly}. We should stress that the EBE method is computationally not the most efficient method since it is designed to avoid singularities by tracking along the surface corresponding to one component of the nonlinear equations. Therefore, a more efficient EBE method will be one of future directions.


\appendix
\section{The detailed calculation of the Jacobian of the map $H_{i}$}\label{Jac}

In this Appendix, we will give the detailed computation for the Jacobian of the map $H_{i}$ in \eqref{mapH} evaluated at $\bm{\mu}^{(i)}$, the solution of ${\bf F}_{i}(\bm{\lambda}_i,\alpha_{i+1},\ldots,\alpha_n)={\bf 0}$. Recall that for $\bm{H}_i=(\bm{H}_{i,1},H_{i,2})$ in \eqref{mapH},
\bes
\bm{H}_{i,1} (\bm{\lambda}_i) &=& \bm{g}_{i}  - \mathbf{F}_{i-1,\bm{\lambda}_{i-1}}(\bm{g}_{i},H_{i,2})^{-1} \mathbf{F}_{i-1}(\bm{g}_{i},H_{i,2}) \nonumber\\
H_{i,2}  (\bm{\lambda}_i) &=& \lambda_{i}-\Big(\frac{\partial F_{i}}{\partial \lambda_{i}}(\bm{\lambda}_{i})\Big)^{-1} F_{i}(\bm{\lambda}_{i}), \nonumber
\ees
where $\bm{g}_{i}:\mathbb{R}^{i-1}\to\mathbb{R}^{i-1}$ is defined as in \eqref{functiong}.

To take another derivative of $\bm{H}_{i,1}$ with respect to $\lambda_j$. We use the fact that if ${\bf F}_{i-1,\bm{\lambda}_{i-1}}$ is a nonsingular matrix, then
\bes
\frac{\partial}{\partial \lambda_j} \Big({\bf F}_{i-1,\bm{\lambda}_{i-1}}\Big)^{-1} = ({\bf F}_{i-1,\bm{\lambda}_{i-1}})^{-1} \frac{\partial {\bf F}_{i-1,\bm{\lambda}_{i-1}}}{\partial \lambda_j} ({\bf F}_{i-1,\bm{\lambda}_{i-1}})^{-1},\nonumber
\ees
and the Hessian $\frac{\partial {\bf F}^*_{i-1,\bm{\lambda}_{i-1}}}{\partial \lambda_j}$ is well-defined, which are the Assumptions~1.2 and 1.3.  We can deduce that for $j=1,\ldots,i$,
\bes
\frac{\partial \bm{H}_{i,1}}{\partial \lambda_j} &=& \frac{\partial \bm{g}_{i}}{\partial \lambda_j} - ({\bf F}_{i-1,\bm{\lambda}_{i-1}})^{-1} ({\bf F}_{i-1,\bm{\lambda}_{i-1}})^{-1} \frac{\partial {\bf F}_{i-1,\bm{\lambda}_{i-1}}}{\partial \lambda_j} ({\bf F}_{i-1,\bm{\lambda}_{i-1}})^{-1} \mathbf{F}_{i-1} \nonumber \\ && - \Big({\bf F}_{i-1,\bm{\lambda}_{i-1}})^{-1} ({\bf F}_{i-1,\bm{\lambda}_{i-1}} \frac{\partial \bm{g}_{i}}{\partial \lambda_j} + \frac{\partial {\bf F}_{i-1}}{\partial \lambda_{i}}\frac{\partial H_{i,2}}{\partial \lambda_j}\Big),\label{A4}\\
\frac{\partial H_{i,2}}{\partial \lambda_j} &=& \frac{\partial \lambda_{i}}{\partial \lambda_j} - \frac{\partial}{\partial \lambda_j}\Big(\frac{\partial F_{i}}{\partial \lambda_{i}}\Big)^{-1}F_{i} - \Big(\frac{\partial F_{i}}{\partial \lambda_{i}}\Big)^{-1} \frac{\partial F_{i}}{\partial \lambda_j}.\label{A5}
\ees
Evaluating these two equations at $\bm{\mu}^{(i)}$ and using the fact that ${\bf F}_{i}^*:={\bf F}_{i}(\bm{\mu}^{(i)})= {\bf 0}$, the second terms in the right-hand-side of \eqref{A4}-\eqref{A5} vanish and we have,
\bes
\frac{\partial \bm{H}_{i,1}^*}{\partial \lambda_j} &=&\frac{\partial \bm{g}^*_{i}}{\partial \lambda_j} -
({\bf F}^*_{i-1,\bm{\lambda}_{i-1}})^{-1} ({\bf F}^*_{i-1,\bm{\lambda}_{i-1}} \frac{\partial \bm{g}^*_{i}}{\partial \lambda_j} + \frac{\partial {\bf F}^*_{i-1}}{\partial \lambda_{i}}\frac{\partial H^*_{i,2}}{\partial \lambda_j}) \nonumber \\ &=&-
({\bf F}^*_{i-1,\bm{\lambda}_{i-1}})^{-1} \Big(\frac{\partial {\bf F}^*_{i-1}}{\partial \lambda_{i}}\frac{\partial H^*_{i,2}}{\partial \lambda_j}\Big),\nonumber\\
\frac{\partial H^*_{i,2}}{\partial \lambda_j} &=& \delta_{j,i} - \Big(\frac{\partial F^*_{i}}{\partial \lambda_{i}}\Big)^{-1} \frac{\partial F^*_{i}}{\partial \lambda_j}.\nonumber
\ees
where $\delta_{j,i}$ is one only if $j=i$ and zero otherwise.

\section*{Acknowledgments}
We thank Rafail Abramov for supplying the two-dimensional density data set for \ref{Ex4}. The BFGS code that we used for comparison in \ref{Ex4} was downloaded from \cite{maxentcode}.

\providecommand{\bysame}{\leavevmode\hbox to3em{\hrulefill}\thinspace}
\providecommand{\MR}{\relax\ifhmode\unskip\space\fi MR }
\providecommand{\MRhref}[2]{%
  \href{http://www.ams.org/mathscinet-getitem?mr=#1}{#2}
}
\providecommand{\href}[2]{#2}


\end{document}